%% file: cances_et_al.tex
\documentclass[10pt]{article}

\usepackage{color} 
\usepackage{epsfig}
\usepackage{amsmath,amssymb,amsfonts,amsthm, esint}
\usepackage{moreverb,rotating,graphics}
\usepackage[T1]{fontenc}
\usepackage[normalem]{ulem}
\usepackage{verbatim}
\usepackage[utf8]{inputenc}

\usepackage{latexsym}
\usepackage{pdflscape}
\usepackage{examplep}
\usepackage{longtable}
\usepackage{graphicx}
\usepackage{amsmath}
\usepackage{booktabs}
\usepackage{memhfixc}
 \usepackage{dsfont}
\usepackage{psfrag}
\usepackage{picture}

\newcommand\dps{\displaystyle }

\newtheorem{theorem}{Theorem}[section]
\newtheorem{proposition}[theorem]{Proposition}
\newtheorem{remark}[theorem]{Remark}
\newtheorem{lemma}[theorem]{Lemma}

\newtheorem{definition}[theorem]{Definition}

\newcommand{\eps}{\varepsilon}

\def\div{{\rm div}}

\def\ZZ{{\mathbb Z}}
\def\RR{{\mathbb R}}
\def\cC{{\mathcal C}}
\def\N{{\mathbb N}}
\def\E{{\mathbb E}}
\def\Z{{\mathbb Z}}
\def\R{{\mathbb R}}

\def\P{{\mathbb P}}

\def\cC{{\cal C}}

\def\cM{{\cal M}}
\def\cF{{\cal F}}

\def\cI{{\cal I}}
\def\cS{{\cal S}}
\def\cA{{\mathbb A}}
\def\cD{{\cal D}}

\def\cJ{{\cal J}}

\def\wlim{\rightharpoonup}

\def\cD{\mathcal{D}}

\def\ccA{\mathcal{A}}

\title{An embedded corrector problem for homogenization. \\ Part I: Theory}
\author{Eric Canc\`es$^{1,3}$, Virginie Ehrlacher$^{1,3}$, Fr\'ed\'eric Legoll$^{2,3}$, Benjamin Stamm$^4$, Shuyang Xiang$^1$
\\
{\footnotesize $^1$ CERMICS, \'Ecole des Ponts ParisTech, 77455 Marne-La-Vall\'ee Cedex 2, France}
\\
{\footnotesize $^2$ Laboratoire Navier, \'Ecole des Ponts ParisTech, 77455 Marne-La-Vall\'ee Cedex 2, France}
\\
{\footnotesize $^3$ Inria Paris, MATHERIALS project-team, 2 rue Simone Iff, CS 42112, 75589 Paris Cedex 12, France}
\\
{\footnotesize $^4$ MATHCCES, Department of Mathematics, RWTH Aachen University, Schinkelstrasse 2, D-52062 Aachen, Germany}
}
\date{\today}

\begin{document}


\maketitle

\begin{abstract}
This article is the first part of a two-fold study, the objective of which is the theoretical analysis and numerical investigation of new approximate corrector problems in the context of stochastic homogenization. We present here three new alternatives for the approximation of the homogenized matrix for diffusion problems with highly-oscillatory coefficients. These different approximations all rely on the use of an {\em embedded} corrector problem (that we previously introduced in~\cite{notre_cras}), where a finite-size domain made of the highly oscillatory material is embedded in a homogeneous infinite medium whose diffusion coefficients have to be appropriately determined. The motivation for considering such embedded corrector problems is made clear in the companion article~\cite{refpartii}, where a very efficient algorithm is presented for the resolution of such problems for particular heterogeneous materials. In the present article, we prove that the three different approximations we introduce converge to the homogenized matrix of the medium when the size of the embedded domain goes to infinity.
\end{abstract}

\section{Introduction}

Let $D \subset \R^d$ be a smooth bounded domain of $\R^d$ (with $d \in \N^\star$), $f \in L^2(D)$ and $(\cA_\varepsilon)_{\varepsilon > 0}$ be a family of uniformly bounded and coercive diffusion matrix fields such that $\cA_\varepsilon$ varies on the characteristic length-scale $\varepsilon > 0$. We consider the family of elliptic problems
\begin{equation}\label{eq:diveps}
u_\eps \in H^1_0(D), \quad -\hbox{\rm div}\left[ \cA_\varepsilon \, \nabla u_\varepsilon \right]=f \ \text{in $D$}.
\end{equation}
When $\varepsilon$ is much smaller than the characteristic size of the domain $D$, problem~\eqref{eq:diveps} is challenging to address from a numerical perspective. In order to obtain a sufficient accuracy, any discretization method indeed needs to resolve the oscillations of $\cA_\varepsilon$, which leads to a discrete problem with a prohibitively large number of degrees of freedom. 

\medskip

It is well-known (see e.g.~\cite{Bensoussan,Cioranescu,Jikov}) that, if $\cA_\varepsilon$ is bounded and bounded away from zero uniformly in $\varepsilon$, problem~\eqref{eq:diveps} admits a homogenized limit. Up to the extraction of a subsequence, that we denote $\eps'$, there exists a homogenized matrix-valued field $\cA^\star \in (L^\infty(D))^{d \times d}$ such that, for any $f \in L^2(D)$, the solution $u_{\eps'}$ to~\eqref{eq:diveps} converges, weakly in $H^1_0(D)$, to $u^\star$, the unique solution to the homogenized equation
\begin{equation}\label{eq:div0}
u^\star \in H^1_0(D), \quad -\mbox{\rm div}\left[ \cA^\star \nabla u^\star \right] = f \mbox{ in $D$}.
\end{equation}
Note that the homogenized matrix, and hence the function $u^\star$, depends in general on the extracted subsequence. 

This setting includes in particular the periodic case, where $\cA_\varepsilon(x)=\cA_{\rm per}(x/\varepsilon)$ for a fixed $\ZZ^d$-periodic function $\cA_{\rm per}$, the quasi-periodic case, where $\cA_\varepsilon(x)=\cA_{\rm q-per}(x/\varepsilon)$ for a fixed quasi-periodic function $\cA_{\rm q-per}$, and the stationary random case (see~\cite{kozlov,papa}), where
$$
\cA_\eps(x)=\cA_{\rm sta}(x/\varepsilon,\omega) \ \text{for some realization $\omega$ of a stationary random function $\cA_{\rm sta}$.}
$$
In these three cases, the convergence of $u_\eps$ to $u^\star$ holds for the whole sequence (and not only up to a subsequence extraction), and the homogenized matrix field $\cA^\star$ is actually equal to a constant and deterministic matrix in the whole domain $D$. Once this homogenized matrix has been determined, problem~\eqref{eq:div0} can be solved by standard numerical techniques with a much lower computational cost than the original problem~\eqref{eq:diveps}.

\medskip

The computation of the homogenized matrix is often a challenging task. In the quasi-periodic case and in the random stationary case, corrector problems posed over the whole space $\R^d$ have to be solved. In practice, approximate corrector problems defined on truncated domains with appropriate boundary conditions (typically periodic boundary conditions) are considered to obtain approximate homogenized diffusion matrices. The larger the size of the truncated domain, the more accurate the corresponding approximation of the homogenized matrix. The use of standard finite element discretizations to tackle these corrector problems may lead to very large discretized problems, whose computational costs can be prohibitive. 

\medskip

In this article, we propose some alternative methods to approximate the homogenized matrix. These are based on the use of an {\em embedded corrector problem} that is again defined over the whole space $\R^d$. In this new problem (see~\eqref{eq:pbbase} below), the diffusion coefficient is equal to $\cA_\eps$ in a bounded domain of typical size $R$, and to a {\em constant} matrix $A_R$ outside this bounded domain, the value of which has to be properly chosen. Our motivation for considering such a family of corrector problems is the following. Recently, a very efficient numerical method has been proposed and developed in the series of works~\cite{Stamm1,Stamm2} in order to solve Poisson problems arising in implicit solvation models. The adaptation of this algorithm, which is based on a boundary integral formulation of the problem, has enabled us to solve these embedded corrector problems in a very efficient way in situations when the considered heterogeneous medium is composed of (possibly polydisperse) spherical inclusions embedded into a homogeneous material (see Fig.~\ref{fig:boules} below). This algorithm will be presented in details in the companion article~\cite{refpartii}. 

\medskip

The choice of the value of the exterior constant diffusion coefficient $A_R$ is instrumental to obtain approximate effective matrices which converge to the exact homogenized matrix when $R$ goes to infinity. In this article, we propose three different approaches to choose the value of the constant exterior diffusion matrix $A_R$ and to define effective matrices from the embedded corrector problem. We prove the convergence of these three approximations to the actual homogenized matrix $\cA^\star$ as $R$ goes to infinity. We also show that a naive choice of $A_R$ leads to an approximate homogenized matrix which {\em does not converge} to the exact homogenized matrix when $R$ goes to infinity.

\medskip

This article is organized as follows. In Section~\ref{sec:existing}, we recall some basic elements on the theory of stochastic homogenization, and review the standard associated numerical methods. The embedded corrector problem mentioned above and the three different approaches we propose to compute effective matrices are presented in Section~\ref{sec:new_defs}. The proofs of consistency of the proposed approximations are collected in Section~\ref{sec:proofs}. Two particular situations, the case of a homogeneous material and the one-dimensional case, for which analytical computations can be performed, are briefly discussed in Section~\ref{sec:justif}. 

The present work complements the earlier publication~\cite{notre_cras}, where we briefly presented our approaches. We provide here a complete and detailed analysis of them. We refer to~\cite{refpartii} for a detailed presentation of the algorithmic aspects along with some numerical illustrations.

\section{Stochastic homogenization: a prototypical example}
\label{sec:existing}

In the sequel, the following notation is used. Let $d\in \N^\star$, $0 < \alpha \leq \beta < +\infty$ and
$$
\cM:= \left\{ A\in \R^{d\times d}, \; A^T = A \ \text{and, for any $\xi \in \RR^d$}, \ \alpha |\xi|^2 \leq \xi^T A \xi \leq \beta |\xi|^2 \right\}.
$$
Let $(e_i)_{1\leq i \leq d}$ be the canonical basis of $\R^d$. Taking $\xi = e_i$ and next $\xi = e_i+e_j$ in the above definition, we see that any $A:=\left( A_{ij} \right)_{1\leq i,j \leq d} \in \cM$ satisfies $|A_{ij}| \leq \beta$ for any $1 \leq i,j \leq d$. We further denote by $\cD(\R^d)$ the set of $C^\infty$ functions with compact supports in $\R^d$.

\medskip

In this section, we briefly recall the well-known homogenization theory in the stationary ergodic setting, as well as standard strategies to approximate the homogenized coefficients. We refer to~\cite{kozlov,papa} for some seminal contributions, to~\cite{engquist-souganidis} for a general, numerically oriented presentation, and to~\cite{Bensoussan,Cioranescu,Jikov} for classical textbooks. We also refer to the review article~\cite{singapour} (and the extensive bibliography contained therein) for a presentation of our particular setting. The stationary ergodic setting can be viewed as a prototypical example of contexts in which the alternative method we propose here for approximating the homogenized matrix can be used.

\subsection{Theoretical setting}

Let $(\Omega, \cF, \P)$ be a probability space and $\dps Q := \left( -\frac{1}{2},\frac{1}{2} \right)^d$. For a random variable $X\in L^1(\Omega, d\P)$, we denote by $\dps \E[X]:= \int_\Omega X(\omega)\, d\P(\omega)$ its expectation value. For the sake of convenience, we restrict the presentation to the case of discrete stationarity, even though the ideas presented here can be readily extended to the case of continuous stationarity. We assume that the group $(\Z^d, +)$ acts on~$\Omega$. We denote by $(\tau_k)_{k\in\Z^d}$ this action, and assume that it preserves the measure $\P$, i.e.
$$
\forall k\in \Z^d, \quad \forall F\in \cF, \quad \P(\tau_k(F)) = \P(F).
$$
We also assume that $\tau$ is ergodic, that is, 
$$
\forall F\in \cF, \quad \left( \forall k\in \Z^d, \; \tau_k F = F \right) \implies \left( \P(F) = 0 \mbox{ or } 1 \right).
$$
A funtion $\cS\in L^1_{\rm loc}\left( \R^d, L^1(\Omega)\right)$ is said to be stationary if
\begin{equation}
\label{eq:stationary}
\forall k\in \Z^d, \quad \cS(x + k, \omega) = \cS(x, \tau_k \omega) \mbox{ for almost all $x\in \R^d$ and almost surely}.
\end{equation}
In that context, the Birkhoff ergodic theorem~\cite{Birkhoff1,Birkhoff2,Birkhoff3} can be stated as follows:
\begin{theorem}
\label{th:Birkhoff}
Let $\cS \in L^\infty\left( \R^d, L^1(\Omega)\right)$ be a stationary function in the sense of~\eqref{eq:stationary}. For $k = (k_1,k_2, \dots, k_d) \in \ZZ^d$, we set $\dps |k|_\infty = \sup_{1\leq i \leq d} |k_i|$. Then, 
$$
\frac{1}{(2N+1)^d}\sum_{|k|_\infty \leq N} \cS(y, \tau_k \omega) \mathop{\longrightarrow}_{N\to +\infty} \E\left[ \cS(y, \cdot) \right] \mbox{ in $L^\infty(\R^d)$, almost surely}.
$$
This implies that
$$
\cS\left( \frac{x}{\varepsilon}, \omega\right) \mathop{\wlim}_{\varepsilon \to 0}^* \E\left[ \frac{1}{|Q|}\int_Q \cS(y, \cdot)\,dy \right] \mbox{ in $L^\infty(\R^d)$, almost surely}.
$$
\end{theorem}

Note that here $|Q|=1$. We kept nevertheless the normalizing factor $|Q|^{-1}$ in the above formula to emphasize that the convergence holds toward the expectation of the mean value over the unit cell of the underlying lattice (here $\Z^d$).

We also recall the definition of $G$-convergence introduced by F.~Murat and L.~Tartar in~\cite{MuratTartar}:

\begin{definition}[$G$-convergence]\label{def:Gconv}
Let $D$ be a smooth bounded domain of $\R^d$. A sequence of matrix-valued functions $\left( \cA^R \right)_{R>0} \subset L^\infty(D, \cM)$ is said to converge in the sense of homogenization (or to $G$-converge) in $D$ to a matrix-valued function $\cA^\star\in L^\infty(D, \cM)$ if, for all $f\in H^{-1}(D)$, the sequence $(u^R)_{R>0}$ of solutions to 
$$
u^R\in H^1_0(D), \quad -\mbox{\div}\left( \cA^R \nabla u^R \right) = f \mbox{ in $\cD'(D)$}
$$
satisfies
$$
\left\{
\begin{array}{l}
  \dps u^R \mathop{\wlim}_{R\to +\infty} u^\star \mbox{ weakly in $H^1_0(D)$},
  \\ \noalign{\vskip 3pt}
\dps \cA^R \nabla u^R \mathop{\wlim}_{R\to +\infty} \cA^\star \nabla u^\star \mbox{ weakly in $L^2(D)$},
\end{array}
\right .
$$
where $u^\star$ is the unique solution to the homogenized equation
$$
u^\star \in H^1_0(D), \quad -\mbox{\rm div}\left( \cA^\star \nabla u^\star \right) = f \mbox{ in $\cD'(D)$}.
$$
\end{definition}

\medskip

The following theorem is a classical result of stochastic homogenization theory (see e.g.~\cite{Jikov}):

\begin{theorem}
\label{th:randhomog}
Let $\cA \in L^\infty(\R^d, L^1(\Omega))$ be such that $\cA(x, \omega) \in \cM$ almost surely and for almost all $x\in \R^d$. We assume that $\cA$ is stationary in the sense of~\eqref{eq:stationary}. For any $R>0$ and $\omega \in \Omega$, we set $\cA^R(\cdot, \omega):= \cA(R\cdot, \omega)$. Then, almost surely, for any arbitrary smooth bounded domain $D\subset \R^d$, the sequence $\left(\cA^R(\cdot, \omega)\right)_{R>0}\subset L^\infty(D; \cM)$ $G$-converges to a \emph{constant and deterministic} matrix $A^\star\in\cM$, which is given by
$$
\forall p \in \R^d, \quad
A^\star p = \E \left[\frac{1}{|Q|} \int_Q \cA(x, \cdot) \left( p + \nabla w_p(x, \cdot) \right) \,dx \right],
$$
where $w_p$ is the unique solution (up to an additive constant) in
$$
\Big\{ v \in L^2_{\rm loc}(\R^d, L^2(\Omega)), \quad \nabla v \in \left( L^2_{\rm unif}(\R^d, L^2(\Omega)) \right)^d \Big\}
$$
to the so-called corrector problem
\begin{equation}
\label{eq:correctorrandom}
\left\{
\begin{array}{l}
  -\mbox{\div}\left( \cA(\cdot, \omega)(p + \nabla w_p(\cdot, \omega))\right) = 0 \mbox{ almost surely in $\cD'(\R^d)$},
  \\ \noalign{\vskip 3pt}
  \nabla w_p \mbox{ is stationary in the sense of~\eqref{eq:stationary}},
  \\ \noalign{\vskip 3pt}
\dps \E\left[ \int_Q \nabla w_p(x, \cdot)\,dx \right] = 0.
\end{array}
\right.
\end{equation}
\end{theorem}
In Theorem~\ref{th:randhomog}, the notation $L^2_{\rm unif}$ refers to the uniform $L^2$ space:
$$
L^2_{\rm unif}(\R^d, L^2(\Omega)) :=\left\{ u \in L^2_{\rm loc}(\R^d;L^2(\Omega)), \quad \sup_{x \in \R^d} \int_{x+(0,1)^d} \|u(y,\cdot)\|_{L^2(\Omega)}^2 \, dy < \infty \right\}.
$$

\medskip

The major difficulty to compute the homogenized matrix $A^\star$ is the fact that the corrector problem~\eqref{eq:correctorrandom} is set over the whole space $\R^d$ and cannot be reduced to a problem posed over a bounded domain (in contrast e.g. to periodic homogenization). This is the reason why approximation strategies yielding practical approximations of $A^\star$ are necessary.

\subsection{Standard numerical practice}

A common approach to approximate $A^\star$ consists in introducing a truncated version of~\eqref{eq:correctorrandom}, see e.g.~\cite{Bourgeat}. For any $R>0$, let us denote $\dps Q_R := \left(-\frac{R}{2},\frac{R}{2} \right)^d$ and
$$
H^1_{\rm per}(Q_R):=\left\{ w \in H^1_{\rm loc}(\R^d), \quad \mbox{$w$ is $R \, \Z^d$-periodic} \right\}.
$$
Observing that $Q_1 = Q$, we also introduce
$$
H^1_{\rm per}(Q):=\left\{ w \in H^1_{\rm loc}(\R^d), \quad \mbox{$w$ is $\Z^d$-periodic} \right\}.
$$
For any $p\in \R^d$, let $\widetilde{w}_p^R(\cdot, \omega)$ be the unique solution in $H^1_{\rm per}(Q_R)/\R$ to
\begin{equation}
\label{eq:correctorrandom-N}
-\mbox{\div}\left( \cA(\cdot, \omega) \left(p + \nabla \widetilde{w}_p^R(\cdot, \omega) \right) \right) = 0 \mbox{ almost surely in $\cD'(\R^d)$}.
\end{equation}
It satisfies the variational formulation 
$$ 
\forall v \in H^1_{\rm per}(Q_R), \quad \int_{Q_R} (\nabla v)^T \cA(\cdot, \omega) \left( p + \nabla \widetilde{w}_p^R(\cdot, \omega) \right) = 0.
$$
The corresponding approximate (or apparent) homogenized matrix $A^{\star,R}(\omega) \in \cM$ is defined by
$$
\forall p \in \R^d, \quad A^{\star,R}(\omega) \, p := \frac{1}{|Q_R|} \int_{Q_R} \cA(\cdot, \omega) \left( p + \nabla \widetilde{w}_p^R(\cdot, \omega) \right).
$$
The matrix $A^{\star,R}(\omega)$ is constant and random. A.~Bourgeat and A.~Piatniski proved in~\cite{Bourgeat} that the sequence of matrices $\left(A^{\star,R}(\omega)\right)_{R>0}$ converges almost surely to $A^\star$ as $R$ goes to infinity. Recent mathematical studies (initiated in~\cite{GloriaOtto1}) by A.~Gloria, F.~Otto and their collaborators, have examined in details the speed of convergence (along with related questions) of $A^{\star,R}(\omega)$ to $A^\star$ (see also~\cite[Theorem 1.3 and Proposition 1.4]{nolen}). Variance reduction techniques have also been introduced to improve the approximation of $A^\star$, see e.g.~\cite{legoll_jcp} for a review.

\begin{remark}
In~\cite{Bourgeat}, A.~Bourgeat and A.~Piatniski also analyzed a truncated corrector problem supplied with homogeneous Dirichlet boundary conditions (in contrast to~\eqref{eq:correctorrandom-N}, where periodic boundary conditions are used) and proved similar convergence results. Likewise, in~\cite{Huet}, C.~Huet introduced a corrector problem supplied with Neumann boundary conditions. 
\end{remark}

\begin{remark}
Besides approximations based on~\eqref{eq:correctorrandom-N}, other techniques have been introduced to approximate $A^\star$. We refer to~\cite{cottereau,lemaire} for optimization-based techniques, to~\cite{mourrat} for an approach based on the heat equation associated to~\eqref{eq:diveps}, and to~\cite{filtrage_sab,filtrage_blanc} for approaches based on filtering. We also mention~\cite{brisard_autres_CL}, where a problem posed on $\R^d$ (which is different from our embedded problem~\eqref{eq:pbbase}) is considered. In a slighly different context, and with a different objective than ours here, the work~\cite{lu_otto} studies the question of optimal artificial boundary condition for random elliptic media.
\end{remark}

\medskip

The proof in~\cite{Bourgeat} relies on the following scaling argument. For any $R>0$, let $\cA^R(\cdot, \omega) := \cA(R\cdot, \omega)$ and $\dps w_p^R(\cdot, \omega) := \frac{1}{R}\widetilde{w}_p^R(R\cdot, \omega)$. Rescaling problem~\eqref{eq:correctorrandom-N}, we obtain that, for any $p\in \R^d$, $w_p^R(\cdot, \omega)$ is the unique solution in $H^1_{\rm per}(Q)/\R$ to
\begin{equation}
\label{eq:correctorrandom-N2}
-\mbox{\div}\left( \cA^R(\cdot, \omega) \left( p + \nabla w_p^R(\cdot, \omega) \right) \right) = 0 \mbox{ almost surely in $\cD'(\R^d)$},
\end{equation}
and that
\begin{equation}
  \label{eq:maison}
A^{\star,R}(\omega) \, p = \frac{1}{|Q|} \int_Q \cA^R(\cdot, \omega) \left( p + \nabla w_p^R(\cdot, \omega) \right). 
\end{equation}
Choosing $w_p^R(\cdot, \omega)$ as the solution to~\eqref{eq:correctorrandom-N2} of zero average, it is easy to see that $\left( w_p^R(\cdot, \omega) \right)_{R>0}$ is bounded in $H^1_{\rm per}(Q)$. In addition, we know that the sequence $\left( \cA^R(\cdot, \omega)\right)_{R>0}$, which belongs to $L^\infty(Q, \cM)$, $G$-converges almost surely to $A^\star$ in $Q$. Using~\cite[Theorem~5.2 page~151]{Jikov} (which is recalled below as Theorem~\ref{th:th1}), we are in position to pass to the limit $R\to +\infty$ in~\eqref{eq:maison} and obtain the desired convergence result. 

\medskip

At this point, we make the following remark. If $\left( \cA^R\right)_{R>0} \subset L^\infty(Q; \cM)$ is a general family of matrix-valued fields which $G$-converges to a constant matrix $A^\star$ as $R$ goes to infinity, one can define for all $R>0$ effective approximate matrices $A^{\star, R}$ as follows. Consider, for any $p\in \R^d$, the unique solution $w_p^R$ in $H^1_{\rm per}(Q)/\R$ to
\begin{equation}
\label{eq:correctorrandom-N3}
-\mbox{\div}\left( \cA^R(p + \nabla w_p^R)\right) = 0 \mbox{ almost surely in $\cD'(\R^d)$},
\end{equation}
and define the matrix $A^{\star, R}$ by
\begin{equation}\label{eq:defper}
\forall p \in \R^d, \quad A^{\star,R} \, p = \frac{1}{|Q|} \int_Q \cA^R \left( p + \nabla w_p^R \right).
\end{equation}
Then, using the same arguments as in the above stationary ergodic case, it can be proven that $\dps A^{\star,R} \mathop{\longrightarrow}_{R\to +\infty} A^\star$. 

\medskip

Solving~\eqref{eq:correctorrandom-N3} by means of standard finite element methods requires the use of very fine discretization meshes, which may lead to prohibitive computational costs. This motivates our work and the alternative definitions of effective matrices that we propose in the next section.

\section{Three alternative definitions of effective matrices}\label{sec:new_defs}

Let $B = B(0,1)$ be the unit open ball of $\R^d$, $\Gamma = \partial B$ and $n(x)$ be the outward pointing unit normal vector at point $x \in \Gamma$. For any measurable subset $E$ of $\R^d$, we denote by $\chi_E$ the characteristic function of $E$.

The embedded corrector problem we define below (see~\eqref{eq:pbbase}) depends on $B$. We note that all the results presented in this article do not use the fact that $B$ is a ball. They can thus be easily extended to the case when $B$ is a general smooth bounded domain of $\R^d$. 

\subsection{Embedded corrector problem}\label{sec:embedded}

In this section, we introduce an {\em embedded corrector} problem, which we will use in the sequel to define new approximations of the homogenized coefficient $A^\star$. 

\medskip

We introduce the vector spaces
\begin{equation}
\label{eq:def_V0}
V:=\left\{v\in L^2_{\rm loc}(\R^d), \ \nabla v \in \left(L^2(\R^d)\right)^d\right\} \quad \mbox{and} \quad V_0:= \left\{ v \in V, \ \int_B v = 0\right\}.  
\end{equation}
The space $V_0$, endowed with the scalar product $\langle \cdot, \cdot \rangle$ defined by
$$
\forall v,w\in V_0, \quad \langle v,w\rangle:= \int_{\R^d} \nabla v \cdot \nabla w,
$$
is a Hilbert space. 

\medskip

For any matrix-valued field $\cA \in L^\infty(B,\cM)$, any constant matrix $A\in \cM$, and any vector $p\in \R^d$, we denote by $w^{\cA,A}_p$ the unique solution in $V_0$ to 
\begin{equation}
\label{eq:pbbase}
-\mbox{\div}\Big( \ccA^{\cA,A} \left( p + \nabla w^{\cA, A}_p \right) \Big) = 0 \mbox{ in $\cD'(\R^d)$},
\end{equation}
where (see Figure~\ref{fig:boules})
$$
\ccA^{\cA,A}(x) := \left| 
\begin{array}{l}
\cA(x) \mbox{ if } x\in B,\\
A \mbox{ if } x\in \R^d \setminus B.
\end{array}\right .
$$
The variational formulation of~\eqref{eq:pbbase} reads as follows: find $w_p^{\cA, A} \in V_0$ such that
\begin{equation}\label{eq:FV}
\forall v\in V_0, \quad \int_B (\nabla v)^T \cA (p + \nabla w_p^{\cA, A}) + \int_{\R^d\setminus B} (\nabla v)^T A \nabla w_p^{\cA, A} - \int_\Gamma (Ap\cdot n) \, v = 0.
\end{equation}
Problem~\eqref{eq:pbbase} is linear and the above bilinear form is coercive in $V_0$. This problem is thus equivalent to a minimization problem (recall that $\cA$ and $A$ are symmetric). The solution $w_p^{\cA, A}$ to~\eqref{eq:pbbase} is equivalently the unique solution to the minimization problem
\begin{equation}
\label{eq:optim1}
w_p^{\cA, A} = \mathop{\mbox{argmin}}_{v\in V_0} J_p^{\cA, A}(v), 
\end{equation}
where
\begin{equation}
\label{eq:optim2}
J^{\cA, A}_p(v) := \frac{1}{|B|} \int_B (p + \nabla v)^T \cA (p + \nabla v) + \frac{1}{|B|} \int_{\R^d\setminus B} (\nabla v)^T A \nabla v - \frac{2}{|B|} \int_{\Gamma} (Ap\cdot n) v.
\end{equation}
We define the map $\cJ_p^{\cA}: \cM \to \R$ by
\begin{equation}
\label{eq:def_curly_J}
\forall A\in \cM, \quad \cJ_p^{\cA}(A) := J_p^{\cA, A} \left( w_p^{\cA, A} \right) = \min_{v\in V_0} J_p^{\cA, A}(v). 
\end{equation}
The linearity of the map $\R^d \ni p \mapsto w_p^{\cA, A} \in V_0$ yields that, for any $A\in \cM$, the map $\R^d\ni p \mapsto \cJ_p^{\cA}(A)$ is quadratic. As a consequence, for all $A\in \cM$, there exists a unique symmetric matrix $G^{\cA}(A)\in \R^{d\times d}$ such that
\begin{equation}
\label{eq:defGA}
\forall p\in \R^d, \quad \cJ^{\cA}_p(A) = p^T G^{\cA}(A) p. 
\end{equation}

\subsection{Motivation of the embedded corrector problem}

For all $R>0$, let us denote by $B_R$ the open ball of $\R^d$ centered at $0$ of radius $R$. We make the following remark, considering, for the sake of illustration, the stationary ergodic setting. Let $\cA(x,\omega)$ be a stationary random matrix-valued field. A simple scaling argument shows that, in this case, for all $R>0$ and $p\in \R^d$, the unique solution $\widetilde{w}_p^{R,\cA,A}(\cdot, \omega)$ in $V$ to 
\begin{equation}\label{eq:equivrand}
-\mbox{div}\Big( \big( \cA(\cdot,\omega) \chi_{B_R} + A (1-\chi_{B_R}) \big) \big( p + \nabla \widetilde{w}_p^{R,\cA,A}(\cdot, \omega) \big) \Big) = 0 \mbox{ in $\cD'(\R^d)$}, \quad \int_{B_R} \widetilde{w}_p^{R,\cA,A}(\cdot, \omega) = 0,
\end{equation}
satisfies $\dps w_p^{\cA^R(\cdot, \omega),A}(\cdot, \omega) = \frac{1}{R} \, \widetilde{w}_p^{R,\cA,A}(R \cdot, \omega)$, where $\cA^R(x,\omega):= \cA\left(Rx, \omega\right)$ for any $x\in B$. Solving embedded corrector problems of the form~\eqref{eq:pbbase} with $\cA = \cA^R(\cdot, \omega)$ in $B$ is then equivalent to solving~\eqref{eq:equivrand}. Figure~\ref{fig:boules} gives an illustration of the matrix-valued field $\ccA^{R, \omega, A}:= \cA(\cdot,\omega) \, \chi_{B_R} + A \, (1-\chi_{B_R})$. 

\medskip

\begin{figure}[h]
\centering
\scalebox{0.6}{\input{./random3.pstex_t}}
\scalebox{0.6}{\input{./random5.pstex_t}}
\caption{Left: field $\cA(x,\omega)$. Right: field $\ccA^{R,\omega,A}(x)$: outside the ball $B_R$, the field $\cA(x,\omega)$ is replaced by a uniform coefficient $A$.\label{fig:boules}}
\end{figure}
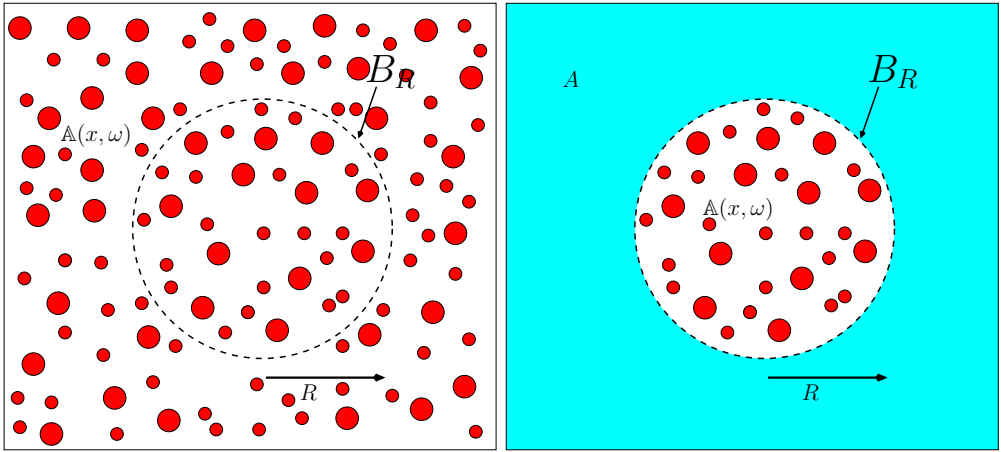

\bigskip

From now on, we consider $\left( \cA^R \right)_{R>0} \subset L^\infty(B; \cM)$ a general family of matrix-valued fields which $G$-converges in the sense of Definition~\ref{def:Gconv} to a {\em constant} matrix $A^\star$ in $B$. Keep in mind that the random stationary ergodic setting provides a prototypical example of such a family of matrix-valued fields. 

\medskip

The rest of the section is devoted to the presentation of different methods for constructing approximate effective matrices, using corrector problems of the form~\eqref{eq:pbbase}. We first present in Section~\ref{sec:def0} a naive definition, which turns out to be non-convergent in general. In the subsequent Sections~\ref{sec:def1}, \ref{sec:def2} and~\ref{sec:def3}, we present three possible choices leading to converging approximations, namely~\eqref{eq:optimisation}, \eqref{eq:def2} and~\eqref{eq:def3}.

\medskip

The motivation for considering problems of the form~\eqref{eq:pbbase} is twofold. First, we show below that the solution $w^{\cA^R, A}_p$ to~\eqref{eq:pbbase} can be used to define consistent approximations of $A^\star$. We refer to Section~\ref{sec:proofs} for the proof that the upcoming approximations~\eqref{eq:optimisation}, \eqref{eq:def2} and~\eqref{eq:def3} converge to $A^\star$ when $R \to \infty$. 

Second, problem~\eqref{eq:pbbase} can be efficiently solved. We recall that, in~\cite{Stamm1, Stamm2}, an efficient numerical method has been introduced to compute the electrostatic interaction of molecules with an infinite continuous solvent medium, based on implicit solvation models. The problem to solve there reads: find $w\in H^1(\Omega)$ solution to
\begin{equation}
\label{eq:pb_stamm}
 - \Delta w = 0 \mbox{ in $\Omega$}, \quad w = g \mbox{ on $\partial \Omega$},
\end{equation}
where $\Omega \subset \R^d$ is a bounded domain composed of the union of a finite but possibly very large number of balls, and $g\in L^2(\partial \Omega)$. As shown in~\cite{Stamm1,Stamm2}, Problem~\eqref{eq:pb_stamm} can be efficiently solved using a numerical approach based on domain decomposition, boundary integral formulation and discretization with spherical harmonics.

\medskip

Inspired by~\cite{Stamm1,Stamm2}, we have developed an efficient algorithm for the resolution of~\eqref{eq:pbbase}, which is somehow similar to the method used for the resolution of~\eqref{eq:pb_stamm}. This algorithm is presented in the companion article~\cite{refpartii}. In short, Problem~\eqref{eq:pbbase} can be efficiently solved using a boundary integral formulation, domain decomposition methods and approximation with spherical harmonics in the case when the matrix-valued field $\cA(x)$ models the diffusion coefficient of a material composed of spherical inclusions embedded in a uniform medium. More precisely, our algorithm is specifically designed to solve~\eqref{eq:pbbase} in the case when, in $B$,
$$
\cA(x)=\left|
\begin{array}{l}
A_{\rm int}^i \mbox{ if } x \in B(x_i,r_i), \ 1 \leq i \leq I,
\\
A_{\rm ext} \mbox{ if } x \in B \setminus \bigcup_{i=1}^I B(x_i, r_i),
\end{array}
\right.
$$
for some $I\in\N^\star$, $A_{\rm int}^i, A_{\rm ext} \in \cM$ for any $1 \leq i \leq I$, $(x_i)_{1\leq i \leq I}\subset B$ and $(r_i)_{1\leq i \leq I}$ some set of positive real numbers such that $\bigcup_{i=1}^I B(x_i,r_i) \subset B$ and $B(x_i,r_i) \cap B(x_j,r_j) = \emptyset$ for all $1\leq i \neq j \leq I$. We have denoted by $B(x_i,r_i) \subset \R^d$ the open ball of radius $r_i$ centered at $x_i$. We refer the reader to~\cite{refpartii} for more details on our numerical method. 

The approach we propose in this article is thus particularly suited for the homogenization of stochastic heterogeneous materials composed of spherical inclusions (see again Figure~\ref{fig:boules}). The properties of the inclusions (i.e. the coefficients $A_{\rm int}^i$), their center $x_i$ and their radius $r_i$ may be random, as long as $\cA$ is stationary. In particular, this algorithm enables to compute very efficiently the effective thermal properties of polydisperse materials.

\subsection{A failed attempt to define a homogenized matrix}\label{sec:def0}

It is of common knowledge in the homogenization community that {\em $G$-convergence is not sensitive to the choice of boundary conditions}, see e.g.~\cite[p.~27]{Allaire}. Thus, at first glance, one could naively think that it would be sufficient to choose a fixed matrix $A\in \cM$, define $w_p^{\cA^R, A}$ for any $p \in \R^d$ and $R>0$ as the unique solution in $V_0$ to~\eqref{eq:pbbase} with $\cA = \cA^R$, and introduce, in the spirit of~\eqref{eq:defper}, the matrix $A_0^R$ defined by
\begin{equation}
\label{eq:def0}
\forall p\in \R^d, \quad A^R_0 p = \frac{1}{|B|} \int_B \cA^R \left( p + \nabla w_p^{\cA^R, A} \right). 
\end{equation}
However, as implied by the following lemma, the sequence $\left( A^R_0 \right)_{R>0}$ defined by~\eqref{eq:def0} {\em does not} converge in general to $A^\star$ as $R$ goes to infinity. Imposing the value of the exterior matrix $A$ in~\eqref{eq:pbbase} is actually much stronger than imposing some (non-oscillatory) boundary conditions on a truncated corrector problem as in~\eqref{eq:correctorrandom-N3}. It turns out that the sequence $\left( A^R_0 \right)_{R>0}$ actually converges as $R$ goes to infinity, but that its limit {\em depends} on the exterior matrix $A$.

\medskip

The following lemma, the proof of which is postponed until Section~\ref{sec:prooflem1}, is not only interesting to guide the intuition. It is also essential in our analysis, in particular for the identification of the limit of $A_0^R$.

\begin{lemma}\label{lem:lem1}
Let $(\cA^R)_{R>0}$ and $(A^R)_{R>0}$ be two sequences such that, for any $R>0$, $\cA^R\in L^\infty(B, \cM)$ and $A^R \in \cM$. We assume that $(\cA^R)_{R>0}$ $G$-converges to a matrix-valued field $\cA^\star\in L^\infty(B,\cM)$ on $B$ and that $(A^R)_{R>0}$ converges to some $A^\infty\in \cM$.
 
For any $R>0$ and $p\in \R^d$, let $w_p^{\cA^R, A^R}$ be the unique solution in $V_0$ to
\begin{equation}
\label{eq:pbn}
-\mbox{\div}\left( {\cal A}^{\cA^R,A^R} \left(p + \nabla w_p^{\cA^R, A^R}\right)\right) = 0 \mbox{ in $\cD'(\R^d)$},
\end{equation}
where 
$$
{\cal A}^{\cA^R,A^R}(x):=\left\{
\begin{array}{l}
 \cA^R(x) \mbox{ if $x\in B$},\\
A^R \mbox{ otherwise}.
\end{array} \right.
$$
Then, the sequence $\left(w^{\cA^R, A^R}_p \right)_{R>0}$ weakly converges in $H^1_{\rm loc}(\R^d)$ to $w_p^{\cA^\star, A^\infty}$, which is the unique solution in $V_0$ to 
\begin{equation}
\label{eq:infty}
-\mbox{\div}\left( {\cal A}^{\cA^\star,A^\infty} \left(p + \nabla w_p^{\cA^\star, A^\infty} \right)\right) = 0 \mbox{ in $\cD'(\R^d)$}, 
\end{equation}
where
$$
{\cal A}^{\cA^\star,A^\infty}(x):= 
\left\{ \begin{array}{l}
\cA^\star(x) \mbox{ if $x\in B$},\\
A^\infty \mbox{ otherwise}.
\end{array}\right. 
$$
Moreover, 
\begin{equation}
\label{eq:infty2}
{\cal A}^{\cA^R,A^R} \left(p + \nabla w_p^{\cA^R, A^R}\right) \mathop{\wlim}_{R\to +\infty} {\cal A}^{\cA^\star,A^\infty} \left(p + \nabla w_p^{\cA^\star, A^\infty}\right) \mbox{ weakly in $L^2_{\rm loc}(\R^d)$}. 
\end{equation}
\end{lemma}

We now briefly show how to use the above lemma to study the limit of $\left( A_0^R\right)_{R>0}$ defined by~\eqref{eq:def0}. From Lemma~\ref{lem:lem1}, we immediately deduce that 
$$
\lim_{R \to \infty} A^R_0 p 
= 
\frac{1}{|B|}\int_B A^\star \left(p + \nabla w_p^{A^\star, A} \right)
=
A^\star p + A^\star \frac{1}{|B|} \int_B \nabla w_p^{A^\star, A}.
$$ 
The above right-hand side is different from $A^\star p$ in general, unless $A=A^\star$, as stated in the following lemma, the proof of which is given in Section~\ref{sec:prooflem2}.

\begin{lemma}\label{lem:lem2}
Let $A^\star, A\in \cM$ and for all $p\in \R^d$, let $w_p^{A^\star, A}$ be the unique solution in $V_0$ to 
\begin{equation}
  \label{eq:taiwan}
- \mbox{\rm div}\left( \ccA^{A^\star, A}\left( p + \nabla w_p^{A^\star, A}\right) \right)  = 0 \mbox{ in $\cD'(\R^d)$},
\end{equation}
where 
$$
\ccA^{A^\star, A}(x):=\left\{
\begin{array}{cc}
A^\star & \mbox{ if $x\in B$},\\
A & \mbox{ otherwise}.
\end{array}
\right .
$$
Then, 
$$
\left[ \forall p\in \R^d, \ \ A^\star p = \frac{1}{|B|}\int_B A^\star \left(p + \nabla w_p^{A^\star, A} \right)\right] \quad \mbox{ if and only if } \quad A = A^\star.
$$
\end{lemma}

\medskip

Thus, we have to find how to define a sequence of constant exterior matrices $(A^R)_{R>0} \subset \cM$ such that problem~\eqref{eq:pbbase} with $\cA = \cA^R$ and $A = A^R$ enables us to introduce converging approximations of $A^\star$. In Sections~\ref{sec:def1}, \ref{sec:def2} and~\ref{sec:def3}, we present three possible choices, which yield three alternative definitions of approximate homogenized matrices that all converge to $A^\star$ when $R \to \infty$.

\subsection{First definition: minimizing the energy of the corrector}
\label{sec:def1}

To gain some intuition, we first recast~\eqref{eq:pbbase} as
$$
-\mbox{\div}\left[ \Big(A + \chi_B (\cA -A) \Big) \, \Big(p + \nabla w_p^{\cA, A} \Big)\right] = 0 \mbox{ in $\cD'(\R^d)$}.
$$
Thus, in this problem, the quantity $\cA-A$ can be seen as a local perturbation in $B$ to the constant homogeneous exterior medium characterized by the diffusion coefficient $A$. In particular, in the case of a perfectly homogeneous infinite medium (when $\cA = A$), the unique solution $w_p^{\cA, A}$ to the above equation is $w_p^{\cA, A}=0$. In the context of homogenization, when the inner matrix-valued coefficient $\cA$ is fixed, a natural idea is then to define the value of the exterior matrix $A$ as {\em the matrix so that the energy $\cJ^{\cA}_p(A)$ of $w_p^{\cA, A}$ (which is always non-positive) is as close to 0 as possible (i.e. as small as possible in absolute value)}. In order to define a more isotropic criterion, we consider the maximization of the quantity $\dps \sum_{i=1}^d \cJ^{\cA}_{e_i}(A)$ rather than $\cJ^{\cA}_p(A)$. This motivates our first definition~\eqref{eq:optimisation}. 

We have the following result, the proof of which is postponed until Section~\ref{sec:proofconcavity}. We recall that $\cJ^{\cA}_p$ and $G^{\cA}$ are defined by~\eqref{eq:def_curly_J} and~\eqref{eq:defGA}.

\begin{lemma}
\label{lem:lemconcavity}
For any $\cA\in L^\infty(B,\cM)$, the function $\dps \mathcal{J}^{\cA}: \cM \ni A \mapsto \sum_{i=1}^d \cJ^{\cA}_{e_i}(A) = \mbox{\rm Tr}\left(G^{\cA}(A)\right)$ is concave. Moreover, when $d \leq 3$, $\mathcal{J}^{\cA}$ is strictly concave. 
\end{lemma}

Since we are interested in practical aspects, we did not investigate the case $d \geq 4$, but we are confident that our arguments could be extended to higher dimensions. 

\medskip

We infer from Lemma~\ref{lem:lemconcavity} that, for any $R>0$, there exists a matrix $A^R_1 \in \cM$ such that
\begin{equation}
\label{eq:optimisation}
A^R_1 \in \mathop{\mbox{argmax }}_{A\in \cM} \sum_{i=1}^d \cJ_{e_i}^{\cA^R}(A) = \mathop{\mbox{argmax }}_{A\in \cM} \mbox{Tr}\left(G^{\cA^R}(A)\right),
\end{equation}
where we recall that $\cA^R = \cA(R \cdot)$. Moreover, in dimension $d \leq 3$, this matrix is unique. Such a matrix $A^R_1$ can be seen as a matrix which minimizes the absolute value of the sum of the energies of the corrector functions $w_{e_i}^{\cA^R,A}$ over all possible $A\in \cM$. Indeed, using the equivalent expression~\eqref{eq:expJ3} of $\cJ_p^{\cA^R}(A)$ given below, we have that
$$
A^R_1 \in \mathop{\mbox{argmin }}_{A\in \cM} \sum_{i=1}^d \left( \int_B \left( \nabla w_{e_i}^{\cA^R, A}\right)^T \cA^R \nabla w_{e_i}^{\cA^R, A} + \int_{\R^d \setminus B} \left( \nabla w_{e_i}^{\cA^R, A}\right)^T A \nabla w_{e_i}^{\cA^R, A} \right).
$$
This provides a justification of the definition of $A^R_1$ by~\eqref{eq:optimisation}. 

\medskip

As shown in Proposition~\ref{prop:prop1} below, $A^R_1$ is a converging approximation of $A^\star$.

\subsection{Second definition: an averaged effective matrix}
\label{sec:def2}

We present here a second natural way to define an effective approximation of the homogenized matrix using the matrix $A^R_1$ defined in the previous section. The idea is to define the matrix $A^R_2 \in \cM$ such that, formally, for any $p \in \RR^d$,
\begin{multline}
  \label{eq:maison2}
\int_B \left[\left(p + \nabla w_p^{\cA^R, A^R_1}\right)^T \cA^R \left(p + \nabla w_p^{\cA^R, A^R_1}\right) - p^T A^R_2 p \right] 
\\
+ \int_{\R^d\setminus B}\left[ \left(p + \nabla w_p^{\cA^R, A^R_1}\right)^T A^R_1 \left(p + \nabla w_p^{\cA^R, A^R_1}\right) - p^T A^R_1 p \right] = 0.
\end{multline}
Formally, we thus ask that the energy of $p \cdot x + w_p^{\cA^R, A^R_1}(x)$ (measured with the energy associated to $\ccA^{\cA^R, A^R_1}(x)$) is equal to the energy of $p \cdot x$ (measured with the energy associated to $\ccA^{A^R_2, A^R_1}(x)$).

Note however that the second term in~\eqref{eq:maison2} is not necessarily well-defined, since we may have $\nabla w_p^{\cA^R, A^R_1} \not\in \left[L^1(\R^d\setminus B)\right]^d$. Formally, the above relation reads
\begin{multline*}
\frac{1}{|B|} \int_B \left(p + \nabla w_p^{\cA^R, A^R_1}\right)^T \cA^R \left(p + \nabla w_p^{\cA^R, A^R_1}\right)
\\
+ \frac{1}{|B|} \int_{\R^d\setminus B} \left(\nabla w_p^{\cA^R, A^R_1}\right)^T A^R_1 \nabla w_p^{\cA^R, A^R_1} 
- 
\frac{2}{|B|} \int_{\R^d\setminus B} (A^R_1 p \cdot n) \, w_p^{\cA^R, A^R_1} 
= 
p^T A^R_2 p, 
\end{multline*}
where now all the terms are well-defined. In view of~\eqref{eq:optim2} and~\eqref{eq:def_curly_J}, the above relation reads
\begin{equation}
\label{eq:def2}
\forall p\in\R^d, \quad p^T A_2^R p = J_p^{\cA^R,A^R_1}\left(w_p^{\cA^R, A^R_1}\right) = \cJ_p^{\cA^R}(A^R_1),
\end{equation}
which implies, in view of~\eqref{eq:defGA}, that
$$
A_2^R = G^{\cA^R}(A^R_1),
$$
where $A^R_1$ is a solution to~\eqref{eq:optimisation}.

\medskip

We prove the following convergence result in Section~\ref{sec:proofprop1}. 

\begin{proposition}\label{prop:prop1}
Let $(\cA^R)_{R>0} \subset L^\infty(B, \cM)$ be a family of matrix-valued fields which $G$-converges in $B$ to a constant matrix $A^\star \in \cM$ as $R$ goes to infinity. 

Then, the sequences of matrices $\left( A^R_1 \right)_{R>0}$ and $\left( A^R_2 \right)_{R>0}$, respectively defined by~\eqref{eq:optimisation} and~\eqref{eq:def2}, satisfy
$$
A^R_1 \mathop{\longrightarrow}_{R\to +\infty} A^\star 
\quad \mbox{ and } \quad
A^R_2 \mathop{\longrightarrow}_{R\to +\infty} A^\star. 
$$
\end{proposition}

\subsection{Third definition: a self-consistent effective matrix}
\label{sec:def3}

We eventually introduce a third definition, inspired by~\cite{Christensen}. Let us assume that, for any $R>0$, there exists a matrix $A^R_3 \in \cM$ such that  
\begin{equation}
\label{eq:def3}
A^R_3  = G^{\cA^R}(A^R_3).
\end{equation}
Such a matrix formally satisfies the following equation (see~\eqref{eq:maison2}): for all $p\in \R^d$,
\begin{multline*}
\int_B \left[\left(p + \nabla w_p^{\cA^R, A^R_3}\right)^T \cA^R \left(p + \nabla w_p^{\cA^R, A^R_3}\right) - p^T A^R_3 p \right] 
\\
+ \int_{\R^d\setminus B}\left[ \left(p + \nabla w_p^{\cA^R, A^R_3}\right)^T A^R_3 \left(p + \nabla w_p^{\cA^R, A^R_3}\right) - p^T A^R_3 p \right] = 0. 
\end{multline*}
Formally, the energy of $p \cdot x + w_p^{\cA^R, A^R_3}(x)$ (measured with the energy associated to $\ccA^{\cA^R, A^R_3}(x)$) is equal to the energy of $p \cdot x$ (measured with the energy associated to $A^R_3$).

\medskip

This third definition also yields a converging approximation of $A^\star$, as stated in the following proposition which is proved in Section~\ref{sec:proofprop2}: 
\begin{proposition}
\label{prop:prop2}
Let $(\cA^R)_{R>0} \subset L^\infty(B, \cM)$ be a family of matrix-valued fields which $G$-converges in $B$ to a constant matrix $A^\star \in \cM$ as $R$ goes to infinity. 

Let us assume that, for any $R>0$, there exists a matrix $A^R_3\in \cM$ satisfying~\eqref{eq:def3}. Then, 
$$
A^R_3 \mathop{\longrightarrow}_{R\to +\infty} A^\star. 
$$
\end{proposition}

\begin{remark}
It is sufficient to assume that there exists a sequence $R_n$ converging to $\infty$ and such that, for any $n \in \N^\star$, there exists a matrix $A^{R_n}_3\in \cM$ satisfying~\eqref{eq:def3}. Then $\dps \lim_{n \to \infty} A^{R_n}_3 = A^\star$. 
\end{remark}
    
In general, we are not able to prove the existence of a matrix $A_3^R$ satisfying~\eqref{eq:def3}. However, the following weaker existence result holds in the case of an isotropic homogenized medium. Its proof is postponed until Section~\ref{sec:proofprop3}.

\begin{proposition}
\label{prop:prop3}
Let $(\cA^R)_{R>0} \subset L^\infty(B, \cM)$ be a family of matrix-valued fields which $G$-converges in $B$ to a constant matrix $A^\star \in \cM$ as $R$ goes to infinity. In addition, assume that $A^\star = a^\star I_d$, where $I_d$ is the identity matrix of $\RR^{d \times d}$.

Then, for any $R>0$, there exists a positive number $a^R_3 \in [\alpha,\beta]$ (which is unique at least in the case when $d \leq 3$) such that
\begin{equation}
\label{eq:spher}
a^R_3 = \frac{1}{d}\mbox{\rm Tr}\left( G^{\cA^R} \left( a^R_3 \, I_d \right) \right).
\end{equation}
In addition, 
\begin{equation}
\label{eq:spher2}
a^R_3 \mathop{\longrightarrow}_{R\to +\infty} a^\star. 
\end{equation}
\end{proposition}
Again, we did not investigate whether the solution to~\eqref{eq:spher} is unique in dimension $d \geq 4$.

\medskip

Note that, since $A^\star = a^\star I_d \in \cM$, we have that $a^\star \in [\alpha,\beta]$. Note also that~\eqref{eq:spher} is weaker than~\eqref{eq:def3}, which would read $a^R_3 \, I_d = G^{\cA^R}(a^R_3 \, I_d)$. However, this weaker result is sufficient to prove that $a^R_3$ is a converging approximation of $a^\star$.

\begin{remark}
In the mechanics literature, other types of approximations have been proposed, based on the analytical solution of the so-called Eshelby problem~\cite{Eshelby}. We refer the reader to the Appendix of~\cite{Thomines} for a pedagogical mathematical introduction to the main methods (including those presented in~\cite{Benveniste,Christensen,Hill,MoriTanaka}) that were derived from the works of Eshelby. For the sake of brevity, we do not detail them here.
\end{remark}
 
\section{Proofs of consistency}\label{sec:proofs}

We collect in this section the proofs of the above propositions. We begin by proving some technical lemmas useful in our analysis. 

\subsection{Preliminary lemmas}

We first recall two classical functional analysis results on the space $V_0$ defined by~\eqref{eq:def_V0}. The first result can be proved using a standard contradiction argument.

\begin{lemma}[Poincar\'e-Wirtinger inequality in $V_0$]\label{lem:Poincare}
For all $r>0$, there exists $K_r>0$ such that
\begin{equation}
\label{eq:Poincare}
\forall v \in V_0, \quad \left\| v \right\|_{L^2(B_r)} \leq K_r \left\| \nabla v \right\|_{L^2(B_r)},
\end{equation}
where $B_r := B(0,r)$ is the open ball of $\R^d$ of radius $r$ and centered at the origin.
\end{lemma}

The next lemma is a straightforward consequence of the continuity of the trace application from $H^1(B)$ to $L^2(\Gamma)$ and of inequality~\eqref{eq:Poincare} for $r=1$.

\begin{lemma}
\label{lem:trace}
There exists $L>0$ such that
\begin{equation}
\label{eq:trace}
\forall v \in V_0, \quad \left\| v \right\|_{L^2(\Gamma)} \leq L \left\| \nabla v \right\|_{L^2(B)}. 
\end{equation}
\end{lemma}

We next recall a classical homogenization result (see e.g.~\cite[Theorem 5.2 p.~151]{Jikov}), which plays a central role in our analysis:

\begin{theorem}
\label{th:th1}
Let $O \subset \R^d$ be an open subset of $\R^d$ and $D$ and $D_1$ two subdomains of $O$ with $D_1\subset D \subset O$. Consider a sequence $(\cA^R)_{R>0} \subset L^\infty(O, \cM)$ and assume that it $G$-converges as $R$ goes to infinity to a matrix-valued function $\cA^\star \in L^\infty(D, \cM)$ in the domain $D$. Besides, let $p\in \R^d$ and let $(w_p^R)_{R>0} \subset H^1(D_1)$ be a sequence of functions which weakly converges (in $H^1(D_1)$) to some $w_p^\infty \in H^1(D_1)$. We assume that
$$
\forall R>0, \quad -\mbox{\div}\left( \cA^R \left( p +\nabla w_p^R \right) \right) = 0 \mbox{ in $\cD'(D_1)$}.
$$   
Then, 
$$
\cA^R \left(p + \nabla w_p^R \right) \wlim \cA^\star \left(p + \nabla w_p^\infty \right) \mbox{ weakly in $L^2(D_1)$}
$$
and $w_p^\infty$ satisfies
$$
-\mbox{\div}\left( \cA^\star \left(p + \nabla w_p^\infty \right) \right) = 0 \mbox{ in $\cD'(D_1)$}. 
$$
\end{theorem}

\medskip

Lastly, the following technical result will be useful in the proofs below:

\begin{lemma}
\label{lem:tech}
Let $(p_i)_{1\leq i \leq d}$ be a basis of $\R^d$. Let $A_1$ and $A_2$ be two constant matrices in $\cM$ such that, for any $1\leq i \leq d$, we have
$$
-\mbox{\div}\Big( \big(A_1 \chi_B + A_2 (1 - \chi_B) \big) p_i \Big) = 0 \mbox{ in $\cD'(\R^d)$}.
$$
Then $A_1 = A_2$.
\end{lemma}

\begin{proof}
For any $\varphi \in \cD(\R^d)$ and any $1\leq i \leq d$, we have
\begin{eqnarray*}
0
&=&
\int_{\RR^d} (\nabla \varphi)^T \left(A_1 \chi_B + A_2 \chi_{\R^d \setminus B} \right) p_i
\\
&=&
\int_B (\nabla \varphi)^T A_1 \, p_i + \int_{\R^d \setminus B} (\nabla \varphi)^T A_2 \, p_i
\\
&=&
\int_\Gamma (A_1 p_i \cdot n) \varphi - \int_\Gamma (A_2 p_i \cdot n) \varphi.
\end{eqnarray*}
Since $\varphi$ is arbitrary, this implies that
$$
\big( (A_1 -A_2) p_i \big) \cdot n(x) = 0 \ \text{on $\Gamma$}.
$$
Hence $(A_1 -A_2) p_i = 0$ for any $1\leq i \leq d$. Since $(p_i)_{1\leq i \leq d}$ is a basis of $\R^d$, we get $A_1 = A_2$. 
\end{proof}

\subsection{Equivalent definitions of $w^{\cA,A}_p$}
\label{sec:equivalent}

We collect here some equivalent definitions of the solution $w_p^{\cA, A}$ to~\eqref{eq:pbbase}. As pointed out above (see~\eqref{eq:FV}), the variational formulation of~\eqref{eq:pbbase} is
\begin{equation}
\label{eq:FVgen}
\forall v \in V_0, \quad \int_B (\nabla v)^T \cA \left(p + \nabla w^{\cA, A}_p \right) + \int_{\R^d \setminus B} \left(\nabla v\right)^T A \nabla w^{\cA, A}_p - \int_{\Gamma} (A p\cdot n) \, v = 0.
\end{equation}
Taking $v = w_p^{\cA, A}$ as a test function in~\eqref{eq:FVgen}, we obtain the following useful relation:
\begin{equation}
\label{eq:rel1}
\int_B \left( \nabla w_p^{\cA, A}\right)^T \cA \nabla w_p^{\cA, A} + \int_{\R^d \setminus B} \left(\nabla w_p^{\cA, A}\right)^T A \nabla w_p^{\cA, A}
=
- \int_B p^T \cA \nabla w_p^{\cA, A} + \int_{\Gamma} (Ap \cdot n) \, w_p^{\cA, A}.
\end{equation}
We recall, as announced in Section~\ref{sec:embedded}, that $w_p^{\cA, A}$ is equivalently the unique solution to the optimization problem~\eqref{eq:optim1}--\eqref{eq:optim2}. We then infer from~\eqref{eq:def_curly_J} and~\eqref{eq:rel1} that
\begin{equation}
\label{eq:expJ3}
\cJ_p^{\cA}(A) = \frac{1}{|B|} \int_B p^T \cA p - \frac{1}{|B|} \int_B \left( \nabla w^{\cA, A}_p\right)^T \cA \nabla w^{\cA, A}_p - \frac{1}{|B|} \int_{\R^d\setminus B} \left( \nabla w^{\cA, A}_p\right)^T A \nabla w^{\cA, A}_p. 
\end{equation}
Equivalently, we also have that
\begin{equation}
\label{eq:expJ2}
\cJ_p^{\cA}(A) = \frac{1}{|B|} \int_B p^T \cA \left(p + \nabla w^{\cA, A}_p \right) - \frac{1}{|B|} \int_{\Gamma} (Ap \cdot n) \, w^{\cA, A}_p.
\end{equation}

\subsection{Proof of Lemma~\ref{lem:lem1}}
\label{sec:prooflem1}

We first show that the sequence $\left( \left\| \nabla w_p^{\cA^R, A^R}\right\|_{L^2(\R^d)}\right)_{R>0}$ is bounded. The weak formulation of~\eqref{eq:pbn} is given by~\eqref{eq:FVgen} with $\cA \equiv \cA^R$ and $A \equiv A^R$. Using~\eqref{eq:rel1} (again with $\cA \equiv \cA^R$ and $A \equiv A^R$), we have
\begin{eqnarray*}
\alpha \left\| \nabla w_p^{\cA^R, A^R}\right\|_{L^2(\R^d)}^2
&\leq& 
\int_B \left(\nabla w_p^{\cA^R, A^R} \right)^T \cA^R \nabla w_p^{\cA^R, A^R}
+
\int_{\R^d \setminus B} \left(\nabla w_p^{\cA^R, A^R} \right)^T A^R \nabla w_p^{\cA^R, A^R}
\\
&=& 
\int_{\Gamma} (A^Rp \cdot n) \, w_p^{\cA^R, A^R} - \int_B p^T \cA^R \left( \nabla w_p^{\cA^R, A^R} \right)
\\
& \leq &
\beta\left( |\Gamma|^{1/2} \left\| w_p^{\cA^R, A^R}\right\|_{L^2(\Gamma)} + |B|^{1/2} \left\| \nabla w_p^{\cA^R, A^R} \right\|_{L^2(B)} \right) 
\\
& \leq &
\beta \left( |\Gamma|^{1/2} L + |B|^{1/2} \right) \left\|\nabla w_p^{\cA^R, A^R} \right\|_{L^2(\R^d)},
\end{eqnarray*}
where, in the last line, we have used~\eqref{eq:trace}. We deduce that, for all $R>0$, 
$$
\left\| \nabla w_p^{\cA^R, A^R} \right\|_{L^2(\R^d)} \leq \frac{\beta}{\alpha}\left( |\Gamma|^{1/2} L + |B|^{1/2} \right).
$$
Let $r>0$. Using~\eqref{eq:Poincare}, we deduce from the above bound that the sequence $\left( \left\| w_p^{\cA^R, A^R} \right\|_{H^1(B_r)} \right)_{R>0}$ is bounded. Therefore, up to the extraction of a subsequence, there exists a function $w^{\infty,r}_p \in H^1(B_r)$ such that
$$
w_p^{\cA^R, A^R} \mathop{\wlim}_{R\to +\infty} w^{\infty,r}_p \mbox{ weakly in $H^1(B_r)$}. 
$$ 
By uniqueness of the limit in the distributional sense, we see that $w^{\infty,r'}_p|_{B_r} = w^{\infty,r}_p$ for any $r' >r$. Thus, there exists a function $w_p^\infty\in H^1_{\rm loc}(\R^d)$ such that, up to the extraction of a subsequence,
\begin{equation}
  \label{eq:maison3}
w_p^{\cA^R, A^R} \mathop{\wlim}_{R\to +\infty} w^{\infty}_p \mbox{ weakly in $H^1_{\rm loc}(\R^d)$}. 
\end{equation}
Moreover, since the sequence $\left( \left\| \nabla w_p^{\cA^R, A^R} \right\|_{L^2(\R^d)} \right)_{R>0}$ is bounded, there exists $W_p^\infty \in \left( L^2(\R^d) \right)^d$ such that (up to the extraction of a subsequence) $\dps \nabla w_p^{\cA^R, A^R} \mathop{\wlim}_{R\to +\infty} W^\infty_p$ weakly in $L^2(\R^d)$. By uniqueness of the limit, we get that $\nabla w^{\infty}_p = W^\infty_p \in \left( L^2(\R^d) \right)^d$. As a consequence, we obtain that $w^\infty_p \in V$. In addition, we obviously have $\dps \int_B w^\infty_p= 0$ and thus $w^\infty_p \in V_0$.

\medskip

At this point, we have shown that, up to the extraction of a subsequence, $w_p^{\cA^R, A^R}$ weakly converges as $R \to \infty$ to $w^\infty_p$ in $H^1(B)$. Furthermore, we know that
$$
-\mbox{\div} \left( \cA^R \left( p + \nabla w_p^{\cA^R, A^R} \right) \right) = 0 \mbox{ in $\cD'(B)$}
$$ 
and that the sequence $\left( \cA^R \right)_{R>0}$ $G$-converges to $\cA^\star$ in $B$. Hence, using Theorem~\ref{th:th1} with the choice $D_1 = B$, we obtain that
\begin{equation}
\label{eq:toto_in}
\cA^R \left( p+ \nabla w_p^{\cA^R, A^R} \right) \wlim \cA^\star \left(p + \nabla w_p^\infty \right) \mbox{ weakly in $L^2(B)$}. 
\end{equation}
For any compact domain $D_1 \subset \R^d \setminus B$, we infer from~\eqref{eq:maison3} that
$$
A^R \left( p+ \nabla w_p^{\cA^R, A^R} \right) \wlim A^\infty \left( p + \nabla w_p^\infty \right) \mbox{ weakly in $L^2(D_1)$}.
$$
This implies that
\begin{equation}
\label{eq:toto_out}
A^R \left( p+ \nabla w_p^{\cA^R, A^R} \right) \wlim A^\infty \left( p + \nabla w_p^\infty \right) \mbox{ weakly in $L^2_{\rm loc}(\R^d \setminus B)$}.
\end{equation}
Collecting~\eqref{eq:toto_in} and~\eqref{eq:toto_out}, we get the claimed convergence~\eqref{eq:infty2}:
$$
{\cal A}^{\cA^R,A^R} \left(p+ \nabla w_p^{\cA^R, A^R} \right) \wlim {\cal A}^{\cA^\star,A^\infty} \left( p + \nabla w_p^\infty \right) \mbox{ weakly in $L^2_{\rm loc}(\R^d)$}.  
$$
Multiplying the above by $\nabla \varphi$, where $\varphi$ is an arbitrary function in ${\cal D}(\RR^d)$, and using~\eqref{eq:pbn}, we deduce~\eqref{eq:infty}. This concludes the proof of Lemma~\ref{lem:lem1}.

\subsection{Proof of Lemma~\ref{lem:lem2}}
\label{sec:prooflem2}

Assume first that $A^\star = A$. Then, for all $p\in\R^d$, $w_p^{A^\star,A} = 0$ is obviously the unique solution in $V_0$ to~\eqref{eq:taiwan}. This yields that $\dps A^\star p = \frac{1}{|B|} \int_B A^\star (p + \nabla w_p^{A^\star, A})$.

\medskip

Conversely, let us now assume that, for all $p\in \R^d$, we have
$$
A^\star p = \frac{1}{|B|} \int_B A^\star (p + \nabla w_p^{A^\star, A}).
$$
Since $A^\star$ is constant and invertible, this implies that $\dps \int_B \nabla w_p^{A^\star, A} = 0$. Multiplying this equation by $p^T A$, we get that $\dps 0 = \int_B p^T A \nabla w_p^{A^\star,A} = \int_\Gamma (A p\cdot n) \, w_p^{A^\star,A}$. We now write~\eqref{eq:rel1} with $\cA \equiv A^\star$:
$$
\int_B (\nabla w_p^{A^\star,A})^T A^\star \nabla w_p^{A^\star,A} + \int_{\R^d \setminus B} (\nabla w_p^{A^\star,A})^T A \nabla w_p^{A^\star,A} = - p^T A^\star \int_B \nabla w_p^{A^\star, A} + \int_\Gamma (A p\cdot n) \, w_p^{A^\star,A} = 0.
$$
We thus get that $\nabla w_p^{A^\star,A} = 0$ in $\R^d$. Hence~\eqref{eq:taiwan} yields that
$$
\forall p\in \R^d, \quad - \mbox{\rm div}\left[ \left(A^\star \chi_B + A (1 - \chi_B)\right)p\right] = 0 \mbox{ in $\cD'(\R^d)$}.
$$
Using Lemma~\ref{lem:tech}, we get that $A = A^\star$. This concludes the proof of Lemma~\ref{lem:lem2}.

\subsection{Proof of Lemma~\ref{lem:lemconcavity}}\label{sec:proofconcavity}

We first prove a technical lemma which will be used to prove the {\em strict} concavity of $\mathcal{J}^{\cA}$ when $d=3$.

\begin{lemma}\label{lem:dipole}
Let $r>0$ and let $S_r$ (respectively $B_r$) be the sphere (respectively the ball) of radius $r$ of $\R^3$ centered at the origin. Let $\sigma \in \cC^\infty(S_r)$ and $\Phi \in \R^{3\times 3}$ be a constant symmetric matrix such that $\mbox{\rm Tr }\Phi = 0$. Assume that
\begin{equation}\label{eq:rel2}
\forall x \in \R^3 \setminus \overline{B_r}, \quad \int_{S_r} \frac{(x-y)^T \Phi (x-y) }{|x-y|^5} \ \sigma(y) \ dy = 0.
\end{equation}
Then, it holds that either $\Phi = 0$ or $\sigma = 0$. 
\end{lemma}

\begin{proof}[Proof of Lemma~\ref{lem:dipole}]
The proof falls in two steps.

\medskip

\noindent
{\bf Step 1.} The first part of the proof consists in showing that the function 
$$
\R^3 \setminus S_r \ni x \mapsto \widetilde V(x) := \int_{S_r} \frac{(x-y)^T \Phi(x-y)}{|x-y|^5} \ \sigma(y) \ dy \in \R
$$
is in fact the restriction to $\R^3 \setminus S_r$ of the electrostatic potential $V$ generated by the singular distribution $\rho \in {\cal E}'(\R^3)$ supported on $S_r$ and defined by 
\begin{equation}\label{eq:defrho}
\forall \psi \in C^\infty(\R^3), \quad \langle \rho,\psi \rangle_{{\cal E}',C^\infty} = \frac{1}{3} \int_{S_r} (\Phi : D^2\psi(y)) \ \sigma(y) \ dy. 
\end{equation}
The distribution $\rho$ can be interpreted as a smooth layer of quadrupoles on $S_r$. The link between $V$ and $\rho$ will be detailed below.

Since $\rho$ defined by~\eqref{eq:defrho} is compactly supported and of order $2$ (for $\Phi \neq 0$), its Fourier transform is analytic, does not grow faster than $|k|^2$ at infinity, and we have
$$
\widehat{\rho}(0) = \frac{1}{(2\pi)^{3/2}} \langle \rho, 1\rangle_{{\cal E}',C^\infty} = 0, \qquad \frac{\partial\widehat{\rho}}{\partial k_j}(0) = - \frac{i}{(2\pi)^{3/2}} \langle \rho, x_j \rangle_{{\cal E}',C^\infty} = 0.
$$
The Poisson equation $-\Delta V = 4\pi \rho$ therefore has a unique solution $V$ belonging to $\cS'(\R^3)$ and vanishing at infinity. We have $\widehat{V} \in L^\infty(\R^3)$ and
$$
\forall k \in \R^3 \setminus \{0\}, \quad \widehat{V}(k) = \frac{4\pi}{|k|^2} \, \widehat{\rho}(k).
$$
Let $\phi \in \cD(\R^3)$ be supported in $\R^3 \setminus S_r$ and $\psi = \phi \star |\cdot|^{-1}$. Note that $\phi \in {\cal S}(\R^3)$, $\psi \in C^\infty(\R^3)$, $\widehat{\psi} \in L^1(\R^3)$, and $|k|^2 \, \widehat{\psi}(k) = 4\pi \, \widehat{\phi}(k)$. However, $\psi \not\in {\cal S}(\R^3)$. We write
$$
\langle V,\phi \rangle_{{\cal D}',{\cal D}}
=
\langle V,\phi \rangle_{{\cal S}',{\cal S}}
=
\left\langle \overline{\widehat{V}}, \widehat{\phi} \right\rangle_{{\cal S}',{\cal S}}
=
\int_{\R^3} \overline{\widehat{V}(k)} \ \widehat{\phi}(k) \, dk
=
\int_{\R^3} \overline{\widehat{\rho}(k)} \ \widehat{\psi}(k) \, dk
=
\langle \rho,\psi \rangle_{{\cal E}',C^\infty}.
$$
For any $y \in S_r$, we have
$$
\psi(y) = \int_{{\rm Supp}(\phi)} \frac{\phi(x)}{|x-y|} \, dx,
$$
hence, for $y \in S_r$,
$$
\frac{\partial^2 \psi}{\partial y_i \partial y_j}(y) = 3 \int_{{\rm Supp}(\phi)} \frac{\phi(x) \, (x_i-y_i) \, (x_j - y_j)}{|x-y|^5} \, dx - \delta_{ij} \int_{{\rm Supp}(\phi)} \frac{\phi(x)}{|x-y|^3} \, dx.
$$
Using next the fact that $\mbox{\rm Tr }\Phi = 0$, we get
\begin{align*}
  \langle V,\phi \rangle_{{\cal D}',{\cal D}}
  =
  \langle \rho,\psi \rangle_{{\cal E}',C^\infty}
  & =
  \frac{1}{3} \int_{S_r} (\Phi : D^2 \psi(y)) \ \sigma(y) \ dy
  \\
  & =
  \int_{{\rm Supp}(\phi)} \left( \int_{S_r} \frac{(x-y)^T \Phi(x-y)}{|x-y|^5} \, \sigma(y) \, dy \right) \phi(x) \, dx
  \\
  &=
  \int_{{\rm Supp}(\phi)} \widetilde{V} \, \phi.
\end{align*}
Therefore $V|_{\R^3 \setminus S_r}=\widetilde{V}$, as claimed above. Furthermore, hypothesis~\eqref{eq:rel2} implies that $V=0$ in $\R^3 \setminus \overline{B_r}$, hence in particular that $V \in {\cal E}'(\R^3)$.

\medskip

\noindent
{\bf Step 2.} Let us denote by ${\cal H}_l$ the vector space of the homogeneous harmonic polynomials of total degree $l$. Recall that ${\rm dim}({\cal H}_l)=2l+1$ and that a basis of ${\cal H}_l$ consists of the functions of the form $(r^l Y_{lm}(\theta,\varphi))_{-l \leq m \leq l}$, where $(r,\theta,\varphi)$ are the usual spherical coordinates and $Y_{lm}$ are the real spherical harmonics. Since $V \in {\cal E}'(\R^3)$, we have
\begin{equation}\label{eq:harmpol}
\forall l \in \N, \quad \forall p_l \in {\cal H}_l, \quad \langle \rho, p_l \rangle_{{\cal E}',C^\infty} = - \frac {1}{4\pi} \langle \Delta V,p_l \rangle_{{\cal E}',C^\infty} = - \frac{1}{4\pi} \langle V, \Delta p_l \rangle_{{\cal E}',C^\infty} = 0.
\end{equation}
We now assume that $\Phi \neq 0$ and we show that $\sigma = 0$. Without loss of generality, we can assume that $\Phi = \mbox{\rm diag}(a_1,a_2,-a_1-a_2)$ with $a_1$ and $a_2$ in $\R_+$ and $a_1a_2 \neq 0$.

For any $l \in \N$, consider the map $L_l : {\cal H}_{l+2} \ni p_{l+2} \mapsto L_l \, p_{l+2} = \Phi : D^2p_{l+2} \in {\cal H}_l$. We are going to prove that $L_l$ is surjective. Any $p_{l+2} \in {\cal H}_{l+2} $ is of the form
$$
p_{l+2}(x_1,x_2,x_3) = \sum_{k=0}^{l+2} x_3^{l+2-k} q_k(x_1,x_2)
$$
where the $q_k$'s are homogeneous polynomials of total degree $k$ on $\R^2$ satisfying
\begin{equation}
  \label{eq:cond1} 
\forall 0 \le k \le l, \qquad \Delta q_{k+2}+(l+2-k)(l+1-k)q_k=0.
\end{equation}
If additionally $p_{l+2} \in \mbox{\rm Ker}(L_l)$, then there also holds
\begin{equation}
    \label{eq:cond2}
 \forall 0 \le k \le l+2, \qquad \lambda \frac{\partial^2q_k}{\partial x_1^2} + \frac{\partial^2q_k}{\partial x_2^2} = 0 \quad \text{with} \quad \lambda = \frac{2a_1+a_2}{a_1+2a_2}.
\end{equation} 
From~\eqref{eq:cond1}, we infer that $p_{l+2}$ is completely determined by $q_{l+1}$ and $q_{l+2}$. From~\eqref{eq:cond2}, we obtain that, for each $0 \le k \le l+2$, $r_k(x_1,x_2) := q_k(\lambda^{1/2}x_1,x_2)$ is a two-dimensional harmonic homogeneous polynomial of order $k$. Consequently, we have 
$$
r_k(x_1,x_2) = \alpha_k \, \mbox{\rm Re}\big( (x_1+ix_2)^k \big) + \beta_k \, \mbox{\rm Im}\big( (x_1+ix_2)^k \big) \quad \text{for some $\alpha_k$ and $\beta_k$ in $\R$.}
$$
An element of $\mbox{Ker}(L_l)$ is therefore completely determined by $\alpha_{k+1}$, $\beta_{k+1}$, $\alpha_{k+2}$ and $\beta_{k+2}$. Hence, $\mbox{\rm dim}({\mbox{\rm Ker}}(L_l))=4$. It follows that 
$$
\mbox{\rm Rank}(L_l)=\mbox{\rm dim}({\cal H}_{l+2})-\mbox{\rm dim}(\mbox{\rm Ker}(L_l))=(2(l+2)+1)-4=2l+1=\mbox{\rm dim}({\cal H}_l).
$$
Therefore $L_l$ is surjective.

For any $l \in \N$ and $q_l \in {\cal H}_l$, there thus exists $p_{l+2} \in {\cal H}_{l+2}$ such that $q_l = L_l p_{l+2}$. We then deduce from~\eqref{eq:defrho} and~\eqref{eq:harmpol} that 
$$
\int_{S_r} q_l(y) \, \sigma(y) \, dy = \int_{S_r} L_l p_{l+2}(y) \, \sigma(y) \, dy = 3 \langle \rho, p_{l+2} \rangle_{{\cal E}',C^\infty} = 0.
$$
Since $(Y_{lm})_{-l \le m \le l}$ is a basis of ${\cal H}_l$, we finally obtain that
$$
\forall l \in \N, \quad \forall -l \le m \le l, \quad  \int_{{\mathbb S}^2} Y_{lm}(y) \, \sigma(ry) \, dy = 0,
$$
where ${\mathbb S}^2$ is the unit sphere of $\R^3$. This implies that $\sigma=0$ and thus concludes the proof of Lemma~\ref{lem:dipole}.
\end{proof}

\medskip

We are now in position to prove Lemma~\ref{lem:lemconcavity}. 

\begin{proof}[Proof of Lemma~\ref{lem:lemconcavity}]
Let $\cA\in L^\infty(B, \cM)$. We first prove that, for all $p \in \R^d$, the function $\cM \ni A \mapsto \cJ_p^{\cA}(A)$ is concave. We next prove its strict concavity. The proof falls in three steps.

\medskip

\noindent
\textbf{Step 1.} The concavity of $\cJ_p^{\cA}$ is a straighforward consequence of~\eqref{eq:optim1}--\eqref{eq:optim2}--\eqref{eq:def_curly_J}: $\cJ_p^{\cA}(A)$ is the minimum of a family of functions that depend on $A$ in an affine way: it is hence concave. Because it will be useful for the proof of strict concavity, we now proceed more quantitatively. We recall that $w_p^{\cA, A}$ is defined by~\eqref{eq:pbbase} or equivalently~\eqref{eq:optim1}. Consider $A_1$ and $A_2$ in $\cM$, $\lambda \in [0,1]$ and $A_\lambda = \lambda A_1 + (1-\lambda) A_2$. We compute that
\begin{align*}
& |B| \, \cJ_p^{\cA}(A_\lambda)
\\
&=
|B| \, J_p^{\cA,A_\lambda}(w_p^{\cA, A_\lambda})
\\
&=
\int_B (p + \nabla w_p^{\cA, A_\lambda})^T \cA (p + \nabla w_p^{\cA, A_\lambda}) + \int_{\R^d\setminus B} (\nabla w_p^{\cA, A_\lambda})^T A_\lambda \nabla w_p^{\cA, A_\lambda} - 2 \int_{\Gamma} (A_\lambda p\cdot n) \, w_p^{\cA, A_\lambda}
\\
&=
\lambda \left( \int_B (p + \nabla w_p^{\cA, A_\lambda})^T \cA (p + \nabla w_p^{\cA, A_\lambda}) + \int_{\R^d\setminus B} (\nabla w_p^{\cA, A_\lambda})^T A_1 \nabla w_p^{\cA, A_\lambda} - 2 \int_{\Gamma} (A_1 p\cdot n) \, w_p^{\cA, A_\lambda} \right)
\\
&+
(1-\lambda) \left( \int_B (p + \nabla w_p^{\cA, A_\lambda})^T \cA (p + \nabla w_p^{\cA, A_\lambda}) + \int_{\R^d\setminus B} (\nabla w_p^{\cA, A_\lambda})^T A_2 \nabla w_p^{\cA, A_\lambda} - 2 \int_{\Gamma} (A_2 p\cdot n) \, w_p^{\cA, A_\lambda} \right)
\\
&= \lambda \, |B| \, J_p^{\cA,A_1}(w_p^{\cA, A_\lambda}) + (1-\lambda) \, |B| \, J_p^{\cA,A_2}(w_p^{\cA, A_\lambda}).
\end{align*}
In view of~\eqref{eq:def_curly_J}, we obtain that
\begin{equation}
  \label{eq:sera_utile1}
\cJ_p^{\cA}(A_\lambda) \geq \lambda \cJ_p^{\cA}(A_1) + (1-\lambda) \cJ_p^{\cA}(A_2),
\end{equation}
which means, as already pointed out above, that the function $\cM \ni A \mapsto \cJ_p^{\cA}(A)$ is concave. Furthermore, since the minimizer of $J_p^{\cA,A}$ is unique for any $A \in \cM$, we get that 
\begin{equation}
  \label{eq:sera_utile2}
  \cJ_p^{\cA}(A_\lambda) = \lambda \cJ_p^{\cA}(A_1) + (1-\lambda) \cJ_p^{\cA}(A_2) \quad \Longrightarrow \quad w_p^{\cA, A_\lambda} = w_p^{\cA, A_1} = w_p^{\cA, A_2}.
\end{equation}

\bigskip

We now prove the strict concavity of $\dps \cJ^{\cA} = \sum_{i=1}^d \cJ^{\cA}_{e_i}$ in the case when $d\leq3$. To this aim, we assume that there exists two matrices $A_1$ and $A_2$ in $\cM$ so that 
\begin{equation}\label{eq:nostrict}
\forall \lambda\in (0,1), \quad \lambda \cJ^{\cA}(A_1) + (1-\lambda)\cJ^{\cA}(A_2) = \cJ^{\cA}\big(\lambda A_1 + (1-\lambda) A_2\big),
\end{equation}
and we are going to show that $A_1 = A_2$.

In view of~\eqref{eq:sera_utile1}, the assumption~\eqref{eq:nostrict} implies that, for any $1 \leq i \leq d$,
$$
\forall \lambda\in (0,1), \quad \lambda \cJ^{\cA}_{e_i}(A_1) + (1-\lambda) \cJ^{\cA}_{e_i}(A_2) = \cJ^{\cA}_{e_i}\big(\lambda A_1 + (1-\lambda) A_2\big),
$$
which implies, in view of~\eqref{eq:sera_utile2}, that
$$
\forall \lambda\in (0,1), \quad w_p^{\cA, \lambda A_1+(1-\lambda) A_2} = w_p^{\cA, A_1} = w_p^{\cA, A_2}.
$$
For the sake of simplicity, we denote this function by $w_i$ in the rest of the proof. It satisfies 
$$
-\mbox{\rm div} \left( A_1 (e_i + \nabla w_i) \right) = -\mbox{\rm div}\left( A_2 (e_i + \nabla w_i) \right) = 0 \mbox{ in $\cD'\big(\R^d \setminus \overline{B}\big)$}. 
$$
Since $A_1$ and $A_2$ are constant matrices, this implies that, for any $1 \leq i \leq d$,
\begin{equation}\label{eq:solwi}
-\mbox{\rm div}\left( A_1 \nabla w_i \right) = -\mbox{\rm div}\left( A_2 \nabla w_i \right) = 0 \mbox{ in $\cD'\big(\R^d \setminus \overline{B}\big)$}. 
\end{equation}
Standard elliptic regularity theory implies that $w_i$ is analytic in $\R^d \setminus \overline{B}$ (see e.g.~\cite[Sec.~2.4 p.~18]{gilbarg2001elliptic}).

\bigskip

\noindent
\textbf{Step 2.} We now proceed by proving that, when $d \leq 3$, equation~\eqref{eq:solwi} implies that
\begin{equation}\label{eq:solwi_bis}
\text{either $A_1$ and $A_2$ are proportional or $w_i$ is a constant function in $\R^d \setminus \overline{B}$}. 
\end{equation}
The case $d=1$ is straightforward. We now consider the case $d=2$. Without loss of generality, we can assume that $A_1 = I_2$ and $A_2$ is diagonal (this can be shown by a linear coordinate transform and the unique continuation principle). If $A_1$ and $A_2$ are not proportional, it follows from~\eqref{eq:solwi} that
$$
\forall x=(x_1,x_2) \in \R^2 \setminus \overline{B}, \qquad \frac{\partial^2 w_i}{\partial x_1^2}(x_1,x_2) = \frac{\partial^2 w_i}{\partial x_2^2}(x_1,x_2) = 0.
$$
This implies that there exists $(a,b,c,d) \in \R^4$ such that $w(x_1,x_2) = a x_1 x_2 + b x_1 + c x_2 + d$ in $\R^2 \setminus \overline{B}$. Since $\nabla w_i \in L^2(\R^2 \setminus \overline{B})$, it follows that $a=b=c=0$, and hence the claim~\eqref{eq:solwi_bis} when $d=2$.

\medskip

We now turn to the case $d=3$, which is more difficult. Let $r>1$ be sufficiently large so that the following two conditions are satisfied:
$$
E:= A_1^{1/2} S_r \subset \R^d \setminus \overline{B} \quad \mbox{and thus also} \quad \R^d \setminus \left( A_1^{1/2} \overline{B_r} \right) \subset \R^d \setminus \overline{B},
$$
where we recall that $S_r$ (respectively $B_r$) is the sphere (respectively open ball) of radius $r$ in $\R^3$. As a consequence of~\eqref{eq:solwi}, there exists a function $\sigma_i \in \cC^\infty(E)$ so that, for all $x \in \R^d \setminus \left( A_1^{1/2} \overline{B_r} \right)$, we have
\begin{equation}
  \label{eq:taiwan2}
w_i(x) = C + \int_{E} G_{A_1}(x-e) \, \sigma_i(e) \, de
\end{equation}
where $C$ is a constant and $G_{A_1}$ is the Green function of the operator $-\mbox{\rm div}\left( A_1 \nabla \cdot \right)$, which reads $\dps G_{A_1}(z) = \frac{1}{4 \pi \, \sqrt{\mbox{\rm det}(A_1)}} \frac{1}{\sqrt{z^T (A_1)^{-1} z}}$ for all $z\in \R^3 \setminus \{0\}$. Using the change of variables $y := A_1^{-1/2} e$, we obtain that there exists a constant $c>0$ so that
$$
w_i(x) = C + c \int_{S_r} \frac{1}{\left| A_1^{-1/2}x-y \right|} \ \sigma_i(A_1^{1/2} y) \, dy.
$$
Let us denote $\Psi:= A_2 - A_1$. For any $x \in \R^d \setminus \left( A_1^{1/2} \overline{B_r} \right)$, it holds that
\begin{align*}
  0& = \mbox{\rm div}_x \left( \Psi \nabla_x w_i(x)\right)
  \\
  &= c \int_{S_r} \mbox{\rm div}_x\left( \Psi \nabla_x \left[ \frac{1}{\left| A_1^{-1/2}x - y \right|} \right] \right) \, \sigma_i(A_1^{1/2} y) \, dy
  \\
  & = c \int_{S_r} \mbox{\rm div}_x \left( - \frac{\Psi A_1^{-1/2}\left( A_1^{-1/2}x - y \right)}{\left| A_1^{-1/2}x - y \right|^3} \right) \, \sigma_i(A_1^{1/2} y) \, dy
  \\
  & = c \int_{S_r} \left[ 3 \, \frac{\left( A_1^{-1/2}x - y\right)^T A_1^{-1/2} \Psi A_1^{-1/2} \left(A_1^{-1/2}x - y\right)}{\left| A_1^{1/2}x - y \right|^5} - \frac{\mbox{\rm Tr}\left( A_1^{-1/2} \Psi A_1^{-1/2} \right)}{{\left| A_1^{-1/2}x-y \right|^3}}\right] \sigma_i(A_1^{1/2} y) \, dy
  \\
& = c \int_{S_r} \left[ \frac{\left(A_1^{-1/2}x -y\right)^T \Phi \left(A_1^{-1/2}x -y\right)}{\left| A_1^{-1/2}x-y \right|^5}\right] \, \sigma_i(A_1^{1/2} y) \, dy,
\end{align*}
where $\Phi := 3 A_1^{-1/2} \Psi A_1^{-1/2} - \mbox{\rm Tr}(A_1^{-1/2} \Psi A_1^{-1/2}) \, I_3$ is a symmetric matrix, the trace of which vanishes. Since this equality holds true for all $x \in \R^d \setminus \left( A_1^{1/2} \overline{B_r} \right)$, it holds that, for all $\overline{x} \in \R^d \setminus \overline{B_r}$,
$$
0 = \int_{S_r} \left[ \frac{\left(\overline{x} -y\right)^T \Phi \left(\overline{x} -y\right)}{| \overline{x}-y |^5}\right] \widehat{\sigma}_i(y)\,dy, 
$$
where for all $y\in S_r$, $\widehat{\sigma}_i(y) = \sigma_i(A_1^{1/2}y)$. Lemma~\ref{lem:dipole} then implies that:
\begin{itemize}
\item either $\widehat{\sigma}_i = 0$, hence $\sigma_i = 0$, which implies, in view of~\eqref{eq:taiwan2}, that $w_i = C$ on $\R^d \setminus \left( A_1^{1/2} \overline{B}_r \right) \subset \R^d \setminus \overline{B}$. Since $w_i$ is analytic in $\R^d \setminus \overline{B}$, we get that $w_i = C$ on $\R^d \setminus \overline{B}$ (this is the unique continuation property for elliptic equations, see e.g.~\cite{protter}).
  \item or $\Phi = 0$. Then $\Psi = A_2 - A_1 = \mu A_1$ for some $\mu \in \R$ and thus $A_1$ and $A_2$ are proportional.
\end{itemize}
This proves the claim~\eqref{eq:solwi_bis} when $d=3$.

\medskip

\noindent
{\bf Step 3.} We have shown in Step 2 that, when $d \leq 3$, \eqref{eq:solwi_bis} holds for any $1 \leq i \leq d$. We now successively consider the two cases of~\eqref{eq:solwi_bis}.

\medskip

\noindent
{\bf Step 3a.} We consider the first possibility in~\eqref{eq:solwi_bis} and assume that $A_1$ and $A_2$ are proportional, that is $A_2 = (1 + \mu) A_1$. We proceed by contradiction and assume that $\mu \neq 0$. Since 
$$
A_1 (\nabla w_i +p) \cdot n = \cA (\nabla w_i +p)\cdot n = A_2 (\nabla w_i +p) \cdot n = (1+\mu) A_1 (\nabla w_i +p) \cdot n \mbox{ on $\Gamma$}
$$
with $\mu \neq 0$, these functions have to be equal to zero on $\Gamma$. The function $u \in H^1(B)$ defined by $u(x) := w_i(x) + p\cdot x$ for all $x \in B$ is then solution to
$$
- \mbox{\rm div}\left(\cA \nabla u\right) = 0 \quad \mbox{ in $B$},
\qquad
\cA \nabla u \cdot n = 0 \quad \mbox{ on $\Gamma$}.
$$
As a consequence, there exists a constant $C\in \R$ such that $u=C$ in $B$, and $w_i(x) = -p\cdot x + C$ for all $x\in B$. In particular, $\nabla w_i +p = 0$ in $B$. Using the variational formulation~\eqref{eq:FVgen} of the embedded corrector problem with test function $w_i$, we get
$$
\int_{\R^d \setminus B} \left(\nabla w_i\right)^T A_1 \nabla w_i = \int_\Gamma (A_1 p\cdot n) \, w_i, 
$$
and 
$$
\int_{\R^d \setminus B} \left(\nabla w_i\right)^T A_2 \nabla w_i = \int_\Gamma (A_2 p\cdot n) \, w_i. 
$$
In view of~\eqref{eq:optim2}, this implies that 
$$
\cJ^{\cA}_{e_i}(A_1) = - \frac{1}{|B|} \int_{\R^d \setminus B} \left(\nabla w_i\right)^T A_1 \nabla w_i
$$
and 
$$
\cJ^{\cA}_{e_i}(A_2) = - \frac{1}{|B|} \int_{\R^d \setminus B} \left(\nabla w_i\right)^T A_2 \nabla w_i = (1+\mu) \cJ^{\cA}_{e_i}(A_1).
$$
Since $\mu \neq 0$, we obtain that $\cJ^{\cA}_{e_i}(A_2) = \cJ^{\cA}_{e_i}(A_1) = 0$, which yields that $\nabla w_i = 0$ in $\R^d \setminus B$. As a consequence, there exists $\widetilde{C} \in \R$ such that $w_i(x) = \widetilde{C}$ for all $x\in \R^d \setminus B$. The continuity of $w_i$ on $\Gamma$ implies that
$$
\forall x\in \Gamma, \quad C - p\cdot x = \widetilde{C},
$$
which yields the desired contradiction. We hence have shown that, if $A_1$ and $A_2$ are proportional, then $A_2 = A_1$.

\medskip

\noindent
{\bf Step 3b.} We next assume that $A_1$ and $A_2$ are not proportional. Then, in view of~\eqref{eq:solwi_bis}, we know that, for any $1 \leq i \leq d$, $w_i$ is a constant function in $\R^d \setminus \overline{B}$, hence $\nabla w_i = 0$ in $\R^d \setminus \overline{B}$. The function $w_i$ satisfies~\eqref{eq:pbbase} for the tensor $\ccA^{\cA,A_1}$, which implies that
$$
\left. n^T \cA (e_i + \nabla w_i) \right|_{\Gamma_-} = \left. n^T A_1 (e_i + \nabla w_i) \right|_{\Gamma_+} = n^T A_1 e_i \quad \text{on $\Gamma$}.
$$
Since $w_i$ also satisfies~\eqref{eq:pbbase} for the tensor $\ccA^{\cA,A_2}$, we have
$$
\left. n^T \cA (e_i + \nabla w_i) \right|_{\Gamma_-} = n^T A_2 e_i \quad \text{on $\Gamma$}.
$$
We hence deduce that $n^T A_1 e_i = n^T A_2 e_i$ on $\Gamma$, hence $A_1 e_i = A_2 e_i$. This holds for any $1 \leq i \leq d$, thus $A_1 = A_2$.

\medskip

This concludes the proof of Lemma~\ref{lem:lemconcavity}.
\end{proof}

\subsection{Proof of Proposition~\ref{prop:prop1}}
\label{sec:proofprop1}

\noindent
{\bf Step 1: $A_1^R$ converges to $A^\star$.} Since $A_1^R \in \cM$, all its coefficients are bounded. Up to the extraction of a subsequence (which we still denote by $(A_1^R)_{R>0}$ for the sake of simplicity), we know that there exists a matrix $A^\infty_1 \in \cM$ such that $\dps \lim_{R \to \infty} A^R_1 = A^\infty_1$. We now prove that $A^\infty_1 = A^\star$, which implies the convergence of the whole sequence $(A^R_1)_{R>0}$ to $A^\star$.

\medskip

Let $p\in\R^d$. It follows from Lemma~\ref{lem:lem1} that $w_p^{\cA^R, A^R_1}$ weakly converges in $H^1_{\rm loc}(\R^d)$ to $w_p^{A^\star, A^\infty_1}$. In addition,
\begin{equation}
\label{eq:lim}
\ccA^{\cA^R,A^R_1} (p + \nabla w_p^{\cA^R, A^R_1}) \wlim \left(A^\star \chi_B + A^\infty_1 (1-\chi_B) \right) \left(p + \nabla w_p^{A^\star, A^\infty_1}\right) \mbox{ weakly in $L^2_{\rm loc}(\R^d)$}. 
\end{equation}
To prove that $A^\infty_1 = A^\star$, we consider a second family of functions of $V_0$, namely $\left(w_p^{\cA^R, A^\star}\right)_{R>0}$. Recall that, for all $R>0$, $w_p^{\cA^R, A^\star}$ is the unique solution in $V_0$ to
$$
-\mbox{\div} \left( \ccA^{\cA^R,A^\star} \left( p + \nabla w_p^{\cA^R, A^\star} \right)\right) = 0 \mbox{ in $\cD'(\R^d)$}. 
$$
Using Lemma~\ref{lem:lem1} again, we obtain that $\left(w_p^{\cA^R, A^\star}\right)_{R>0}$ weakly converges in $H^1_{\rm loc}(\R^d)$ to $w_p^{A^\star, A^\star} = 0$. Furthermore, we have 
\begin{equation}
\label{eq:limchek}
\ccA^{\cA^R,A^\star} \left( p + \nabla w_p^{\cA^R, A^\star} \right) \wlim A^\star \left( p + \nabla w_p^{A^\star, A^\star} \right) = A^\star p \mbox{ weakly in $L^2_{\rm loc}(\R^d)$}.
\end{equation}
Since $A^R_1$ is (the unique) solution to~\eqref{eq:optimisation}, we have
$$
\sum_{i=1}^d \cJ^{\cA^R}_{e_i}(A^\star) \leq \sum_{i=1}^d \cJ^{\cA^R}_{e_i}(A^R_1), 
$$
which reads, using~\eqref{eq:expJ2}, as
\begin{multline}
\label{eq:genial}
\sum_{i=1}^d \int_B e_i^T \cA^R \left( e_i + \nabla w^{\cA^R, A^\star}_{e_i} \right) - \int_{\Gamma} (A^\star e_i \cdot n) \, w^{\cA^R, A^\star}_{e_i} \\
\leq \sum_{i=1}^d \int_B e_i^T \cA^R \left( e_i + \nabla w_{e_i}^{\cA^R, A^R_1} \right) - \int_{\Gamma} ( A_1^R e_i \cdot n) \, w_{e_i}^{\cA^R, A^R_1}.
\end{multline}
We wish to pass to the limit $R \to \infty$ in this inequality. Using~\eqref{eq:lim} and~\eqref{eq:limchek}, we first have, for any $p \in \R^d$,
\begin{equation}
\label{eq:exc1}
\int_B p^T \cA^R \left( p + \nabla w_p^{\cA^R, A^R_1} \right) \mathop{\longrightarrow}_{R\to +\infty} \int_B p^T A^\star \left( p + \nabla w_p^{A^\star, A^\infty_1} \right) 
\end{equation}
and
\begin{equation}
\label{eq:exc2}
\int_B p^T \cA^R \left( p + \nabla w_p^{\cA^R, A^\star} \right) \mathop{\longrightarrow}_{R\to +\infty} \int_B p^T A^\star p.
\end{equation}
Second, we know that $\widetilde{w}_p^{\cA^R, A^\star}$ (respectively $\widetilde{w}_p^{\cA^R, A^R_1}$) weakly converges in $H^1(B)$ to $\widetilde{w}_p^{A^\star, A^\star} = 0$ (respectively to $\widetilde{w}_p^{A^\star, A^\infty_1}$). The compactness of the trace operator from $H^1(B)$ to $L^2(\Gamma)$ yields that these convergences also hold strongly in $L^2(\Gamma)$. Thus,
\begin{equation}\label{eq:exc3}
\int_\Gamma (A^\star p \cdot n) \, w_p^{\cA^R, A^\star} \mathop{\longrightarrow}_{R\to +\infty} 0
\end{equation}
and
\begin{equation}\label{eq:exc4}
\int_\Gamma (A_1^R p \cdot n) \, w_p^{\cA^R, A^R_1} \mathop{\longrightarrow}_{R\to +\infty} \int_\Gamma (A^\infty_1 p \cdot n) \, w_p^{A^\star, A^\infty_1}.
\end{equation}
Collecting~\eqref{eq:exc1}, \eqref{eq:exc2}, \eqref{eq:exc3} and~\eqref{eq:exc4}, we are in position to pass to the limit $R \to \infty$ in~\eqref{eq:genial}, and deduce that
\begin{equation}
\label{eq:ineg}
\sum_{i=1}^d \int_B e_i^T A^\star e_i \leq \sum_{i=1}^d \int_B e_i^T A^\star \left( e_i + \nabla w_{e_i}^{A^\star, A^\infty_1} \right) - \int_\Gamma (A_1^\infty e_i \cdot n) \, w_{e_i}^{A^\star, A^\infty_1}.
\end{equation}
In view of~\eqref{eq:rel1}, we have that, for all $p\in\R^d$,
\begin{multline*}
\int_B p^T A^\star \nabla w_p^{A^\star, A^\infty_1} - \int_\Gamma (A_1^\infty p \cdot n) \, w_p^{A^\star, A^\infty_1}
\\
= 
- \int_{\R^d \setminus B} \left( \nabla w_p^{A^\star, A^\infty_1} \right)^T A^\infty_1 \nabla w_p^{A^\star, A^\infty_1} - \int_B \left( \nabla w_p^{A^\star, A^\infty_1} \right)^T A^\star \nabla w_p^{A^\star, A^\infty_1},
\end{multline*}
which implies that
\begin{multline*}
\int_B p^T A^\star \left( p + \nabla w_p^{A^\star, A^\infty_1} \right) - \int_\Gamma (A_1^\infty p \cdot n) \, w_p^{A^\star, A^\infty_1}
\\
= \int_B p^T A^\star p - \int_{\R^d \setminus B} \left( \nabla w_p^{A^\star, A^\infty_1} \right)^T A^\infty_1 \nabla w_p^{A^\star, A^\infty_1} - \int_B \left( \nabla w_p^{A^\star, A^\infty_1} \right)^T A^\star \nabla w_p^{A^\star, A^\infty_1}.
\end{multline*}
Thus, \eqref{eq:ineg} yields that 
$$
0 \leq - \sum_{i=1}^d \left[ \int_B \left( \nabla w_{e_i}^{A^\star, A^\infty_1} \right)^T A^\star \nabla w_{e_i}^{A^\star, A^\infty_1} + \int_{\R^d \setminus B}  \left( \nabla w_{e_i}^{A^\star, A^\infty_1} \right)^T A^\infty_1 \nabla w_{e_i}^{A^\star, A^\infty_1}\right],
$$
which implies that $\nabla w_{e_i}^{A^\star, A^\infty_1} = 0$ on $\RR^d$ for all $1\leq i \leq d$. As a consequence, for all $1\leq i \leq d$, 
$$
- \mbox{\rm div}\left[ \left(A^\star \chi_B + A_1^\infty(1-\chi_B)\right) e_i\right] = 0 \mbox{ in $\cD'(\R^d)$}.
$$
In view of Lemma~\ref{lem:tech}, this implies that $A^\infty_1 = A^\star$ and concludes the proof of the first assertion of Proposition~\ref{prop:prop1}. 

\medskip

\noindent
{\bf Step 2: $A_2^R$ converges to $A^\star$.} Recall that $A^R_2$ is defined, following~\eqref{eq:def2}, by $p^T A^R_2 p = \cJ^{\cA^R}_p(A_1^R)$. Using~\eqref{eq:expJ2} and the above arguments, we see that
$$
\lim_{R \to \infty} \cJ^{\cA^R}_p(A_1^R) 
=
\frac{1}{|B|} \int_B p^T A^\star \left( p + \nabla w_p^{A^\star, A^\infty_1} \right) - \frac{1}{|B|} \int_\Gamma (A^\infty_1 p \cdot n) \, w_p^{A^\star, A^\infty_1}.
$$
Since $w_p^{A^\star, A^\infty_1} = w_p^{A^\star, A^\star} = 0$, we get that $\dps \lim_{R \to \infty} \cJ^{\cA^R}_p(A_1^R) = p^T A^\star p$. For any $p \in \RR^d$, we thus have $\dps \lim_{R \to \infty} p^T A^R_2 p = p^T A^\star p$, hence $\dps \lim_{R \to \infty} A^R_2 = A^\star$. This concludes the proof of the second assertion of Proposition~\ref{prop:prop1}.

\subsection{Proof of Proposition~\ref{prop:prop2}}
\label{sec:proofprop2}

Since $A^R_3\in \cM$, all its coefficients are bounded. Hence, up to the extraction of a subsequence (that we still denote by $\left( A^R_3\right)_{R>0}$ to simplify the notation), there exists a matrix $A^\infty_3 \in \cM$ such that $\dps A^R_3 \mathop{\longrightarrow}_{R\to +\infty} A^\infty_3$. We show that $A^\infty_3 = A^\star$.

Let $p\in\RR^d$. Recall that, for all $R>0$, $w_p^{\cA^R, A^R_3}$ is the unique solution in $V_0$ to
$$
-\mbox{\div}\left( \ccA^{\cA^R,A^R_3} \left(p + \nabla w_p^{\cA^R, A^R_3}\right)\right) = 0 \mbox{ in $\cD'(\R^d)$}.
$$
Using Lemma~\ref{lem:lem1}, we have that $\left(w_p^{\cA^R, A^R_3}\right)_{R>0}$ weakly converges in $H^1_{\rm loc}(\R^d)$ to $w_p^{A^\star, A^\infty_3}$, which is the unique solution in $V_0$ to
\begin{equation}
\label{eq:lim2}
-\mbox{\div}\Big( \left(A^\star \chi_B + A^\infty_3(1-\chi_B) \right) \left(p + \nabla w_p^{A^\star, A^\infty_3} \right)\Big) = 0 \mbox{ in $\cD'(\R^d)$}.
\end{equation}
Using~\eqref{eq:defGA} and~\eqref{eq:expJ2}, we see that
$$
p^T G^{\cA^R}(A^R_3) p 
= 
\cJ^{\cA^R}_p(A^R_3)
= 
\frac{1}{|B|} \int_B p^T \cA^R \left( p + \nabla w_p^{\cA^R, A^R_3} \right) - \frac{1}{|B|} \int_\Gamma \left( A^R_3 p \cdot n \right) \, w_p^{\cA^R, A^R_3}.
$$
Using Lemma~\ref{lem:lem1} and arguing as in the proof of Proposition~\ref{prop:prop1}, we deduce that
$$
\lim_{R \to \infty} p^T G^{\cA^R}(A^R_3) p 
=
\frac{1}{|B|} \int_B p^T A^\star \left( p + \nabla w_p^{A^\star, A^\infty_3} \right) - \frac{1}{|B|} \int_\Gamma \left( A^\infty_3 p\cdot n \right) \, w_p^{A^\star, A^\infty_3}.
$$
Passing to the limit $R \to \infty$ in~\eqref{eq:def3}, we hence get that
\begin{equation}
\label{eq:genial2}
p^T A^\infty_3 p 
= 
\frac{1}{|B|} \int_B p^T A^\star \left( p + \nabla w_p^{A^\star, A^\infty_3} \right) - \frac{1}{|B|} \int_\Gamma \left( A^\infty_3 p \cdot n \right) \, w_p^{A^\star, A^\infty_3}.
\end{equation}
Using the relation~\eqref{eq:rel1} for the problem~\eqref{eq:lim2}, we have that
\begin{multline*}
\int_B \left( \nabla w_p^{A^\star, A^\infty_3} \right)^T A^\star \nabla w_p^{A^\star, A^\infty_3} + \int_{\R^d \setminus B} \left( \nabla w_p^{A^\star, A^\infty_3} \right)^T A^\infty_3 \nabla w_p^{A^\star, A^\infty_3} 
\\
= - \int_B p^T A^\star \nabla w_p^{A^\star, A^\infty_3} + \int_\Gamma \left( A^\infty_3 p \cdot n \right) \, w_p^{A^\star, A^\infty_3}.
\end{multline*}
We thus deduce from~\eqref{eq:genial2} that
\begin{equation}
\label{eq:relation}
p^T A^\infty_3 p = p^T A^\star p - \frac{1}{|B|} \int_B \left( \nabla w_p^{A^\star, A^\infty_3} \right)^T A^\star \nabla w_p^{A^\star, A^\infty_3} - \frac{1}{|B|} \int_{\R^d \setminus B} \left( \nabla w_p^{A^\star, A^\infty_3} \right)^T A^\infty_3 \nabla w_p^{A^\star, A^\infty_3}.
\end{equation}
This implies that 
\begin{equation}
\label{eq:genial3}
\text{$A^\infty_3 \leq A^\star$ in the sense of symmetric matrices.}
\end{equation} 
In addition, we infer from~\eqref{eq:relation} that
\begin{align}
&
\int_B \left( p + \nabla w_p^{A^\star, A^\infty_3} \right)^T A^\star p 
\nonumber
\\
&= 
|B| \, p^T A^\star p + \int_B p^T A^\star \nabla w_p^{A^\star, A^\infty_3}
\nonumber
\\
&=
|B| \, p^T A_3^\infty p + \int_B \left( p + \nabla w_p^{A^\star, A^\infty_3} \right)^T A^\star \nabla w_p^{A^\star, A^\infty_3} + \int_{\R^d \setminus B} \left( \nabla w_p^{A^\star, A^\infty_3} \right)^T A^\infty_3 \nabla w_p^{A^\star, A^\infty_3}.
\label{eq:genial4}
\end{align}
The variational formulation of~\eqref{eq:lim2}, tested with the test function $w_p^{A^\star, A^\infty_3}$, yields
$$
0 
=
\int_B \left( p + \nabla w_p^{A^\star, A^\infty_3} \right)^T A^\star \nabla w_p^{A^\star, A^\infty_3} - \int_\Gamma \left( A^\infty_3 p \cdot n \right) \, w_p^{A^\star, A^\infty_3}
+ \int_{\R^d \setminus B} \left( \nabla w_p^{A^\star, A^\infty_3} \right)^T A^\infty_3 \nabla w_p^{A^\star, A^\infty_3}.
$$
Subtracting twice the above relation from~\eqref{eq:genial4}, we get
\begin{multline*}
\int_B \left( p + \nabla w_p^{A^\star, A^\infty_3} \right)^T A^\star p
=
|B| \, p^T A_3^\infty p - \int_B \left( p + \nabla w_p^{A^\star, A^\infty_3} \right)^T A^\star \nabla w_p^{A^\star, A^\infty_3} \\ - \int_{\R^d \setminus B} \left( \nabla w_p^{A^\star, A^\infty_3} \right)^T A^\infty_3 \nabla w_p^{A^\star, A^\infty_3} + 2 \int_\Gamma \left( A^\infty_3 p \cdot n \right) \, w_p^{A^\star, A^\infty_3},
\end{multline*}
which we recast as
\begin{eqnarray}
0 
&=& \int_B \left( p + \nabla w_p^{A^\star, A^\infty_3} \right)^T A^\star \left( p + \nabla w_p^{A^\star, A^\infty_3} \right) - 2 \int_\Gamma \left( A^\infty_3 p \cdot n \right) \, w_p^{A^\star, A^\infty_3}
\nonumber
\\
& & \qquad + \int_{\R^d \setminus B} \left( \nabla w_p^{A^\star, A^\infty_3} \right)^T A^\infty_3 \nabla w_p^{A^\star, A^\infty_3} - |B| \, p^T A^\infty_3 p 
\nonumber
\\
& \geq & 
\int_B \left( p + \nabla w_p^{A^\star, A^\infty_3} \right)^T A^\infty_3 \left( p + \nabla w_p^{A^\star, A^\infty_3} \right) - 2 \int_\Gamma \left( A^\infty_3 p \cdot n \right) \, w_p^{A^\star, A^\infty_3}
\nonumber
\\
& & \qquad + \int_{\R^d \setminus B} \left( \nabla w_p^{A^\star, A^\infty_3} \right)^T A^\infty_3 \nabla w_p^{A^\star, A^\infty_3} - |B| \, p^T A^\infty_3 p,
\label{eq:genial6}
\end{eqnarray}
where we have eventually used~\eqref{eq:genial3}. We now define, for any $v\in V_0$, 
$$
\cI(v) := \frac{1}{2}\int_B (p + \nabla v)^T A^\infty_3 (p + \nabla v) - \int_\Gamma (A^\infty_3 p \cdot n) \, v + \frac{1}{2} \int_{\R^d \setminus B} (\nabla v)^T A^\infty_3 \nabla v - \frac{1}{2} |B| \, p^T A^\infty_3 p.
$$
The unique solution $v_0 \in V_0$ to the minimization problem
$$
v_0 = \mathop{\mbox{argmin}}_{v\in V_0} \cI(v)
$$
satisfies
$$
-\mbox{\div} \Big( A^\infty_3 (p + \nabla v_0) \Big) = 0 \mbox{ in $\cD'(R^d)$},
$$  
and therefore is simply $v_0 = 0$. Thus,
\begin{equation}
\label{eq:genial5}
\forall v \in V_0, \quad \cI(v) \geq \cI(v_0) = 0.
\end{equation}
We recast~\eqref{eq:genial6} as
$$
0 \geq 2 \, \cI(\widetilde{w}_p^{A^\star, A^\infty_3}).
$$ 
Together with~\eqref{eq:genial5}, the above inequality implies that $w_p^{A^\star, A^\infty_3}$ is the unique minimizer of $\cI$ on $V_0$, hence $w_p^{A^\star, A^\infty_3} = 0$. This results holds for all $p\in\R^d$. In view of~\eqref{eq:lim2} and Lemma~\ref{lem:tech}, we thus obtain that $A^\infty_3 = A^\star$. This concludes the proof of Proposition~\ref{prop:prop2}.

\subsection{Proof of Proposition~\ref{prop:prop3}}
\label{sec:proofprop3}

\noindent
\textbf{Step 1: Proof of~\eqref{eq:spher}.} For all $R>0$, $\cA^R \in L^\infty(B; \cM)$, hence, for any $v \in V_0$, we have
$$
J_p^{\alpha I_d,A}(v) \leq J_p^{\cA^R,A}(v) \leq J_p^{\beta I_d,A}(v)
$$
and therefore
$$
p^T G^{\alpha I_d}(A) p \leq p^T G^{\cA}(A) p \leq p^T G^{\beta I_d}(A) p.
$$
Thus, for any $A \in \cM$, we have
\begin{equation}
\label{eq:h1}
\mbox{Tr}\left( G^{\alpha I_d}(A)\right) \leq \mbox{Tr}\left( G^{\cA^R}(A)\right) \leq \mbox{Tr}\left( G^{\beta I_d}(A)\right).
\end{equation}
For any $\gamma \in [\alpha,\beta]$, we introduce $\dps f_{\cA^R}(\gamma) = \frac{1}{d}\mbox{Tr}\left( G^{\cA^R}(\gamma I_d)\right) - \gamma$. Satisfying~\eqref{eq:spher} amounts to finding $a^R_3 \in [\alpha,\beta]$ such that $f_{\cA^R}(a^R_3)=0$. Introducing $\dps f_\alpha(\gamma) = \frac{1}{d}\mbox{Tr}\left( G^{\alpha I_d}(\gamma I_d)\right) - \gamma$ and likewise for $f_\beta(\gamma)$, we deduce from~\eqref{eq:h1} that
\begin{equation}
\label{eq:h2}
\forall \gamma \in [\alpha,\beta], \quad f_\alpha(\gamma) \leq f_{\cA^R}(\gamma) \leq f_\beta(\gamma).
\end{equation}
To proceed, we note that we have an explicit expression of $f_\alpha(\gamma)$, using the explicit solution to Eshelby's problem~\cite{Eshelby}. Indeed, for any $1\leq i \leq d$ and any $\alpha,\gamma>0$, the solution $w_{e_i}^{\alpha I_d,\gamma I_d}$ to~\eqref{eq:pbbase} with $A = \gamma I_d$ et $\cA(x) = \alpha I_d$ on $\RR^d$ is given by
$$
w_{e_i}^{\alpha I_d,\gamma I_d}(x) = \left\{
\begin{array}{l}
C(\alpha, \gamma) \, x_i \mbox{ if $|x|\leq 1$},
\\ \noalign{\vskip 5pt}
\dps C(\alpha, \gamma) \, \frac{x_i}{|x|^d} \mbox{ if $|x|\geq 1$},
\end{array} \right.
\ \text{with} \ C(\alpha, \gamma) = \frac{\gamma - \alpha}{(d-1)\gamma + \alpha}.
$$
With~\eqref{eq:defGA}, \eqref{eq:expJ2} and the above expression, we easily obtain that
$$
\mbox{Tr} \left( G^{\alpha I_d}(\gamma I_d) \right)
=
\sum_{i=1}^d \cJ_{e_i}^{\alpha I_d}(\gamma I_d) 
= 
d \, \left( \alpha + (\alpha-\gamma) C(\alpha, \gamma) \right),
$$
hence
$$
\frac{1}{d}\mbox{Tr}\left( G^{\alpha I_d}(\gamma I_d) \right) 
= 
\alpha + (\alpha-\gamma) C(\alpha, \gamma) ,
$$
and thus
$$
f_\alpha(\gamma) 
= 
(\alpha-\gamma) \Big( C(\alpha, \gamma) + 1 \Big)
=
(\alpha-\gamma) \frac{d\gamma}{(d-1)\gamma+\alpha}.
$$
We see that, when $\gamma \in [\alpha,\beta]$, we have $f_\alpha(\gamma) \leq 0$ and the equation $f_\alpha(\gamma) = 0$ has a unique solution, $\gamma=\alpha$. Likewise, when $\gamma \in [\alpha,\beta]$, we have $f_\beta(\gamma) \geq 0$ and the equation $f_\beta(\gamma) = 0$ has a unique solution, $\gamma=\beta$. The bound~\eqref{eq:h2} implies that there exists $a^R_3 \in [\alpha,\beta]$ such that $f_{\cA^R}(a^R_3)=0$. This proves~\eqref{eq:spher}.

Besides, in the case when $d \leq 3$, Lemma~\ref{lem:lemconcavity} implies that, for any $R>0$, $f_{\cA^R}$ is strictly concave. This yields the uniqueness of $a^R_3$ when $d \leq 3$.

\bigskip

\noindent
\textbf{Step 2: Proof of~\eqref{eq:spher2}.} We follow the same arguments as in the beginning of the proof of Proposition~\ref{prop:prop2}. Since $a^R_3 \in [\alpha,\beta]$, we know that, up to the extraction of a subsequence (that we still denote by $\left( a^R_3\right)_{R>0}$ to simplify the notation), there exists $a^\infty_3 \in [\alpha, \beta]$ such that $\dps a^R_3 \mathop{\longrightarrow}_{R\to +\infty} a^\infty_3$. 

Passing to the limit $R \to \infty$ in~\eqref{eq:spher}, we get that
\begin{equation}
\label{eq:genial2_bis}
d \, a^\infty_3 
= 
\sum_{i=1}^d \frac{a^\star}{|B|} \int_B e_i^T \left( e_i + \nabla w_{e_i}^{a^\star I_d, a^\infty_3 I_d} \right) - \frac{a^\infty_3}{|B|} \int_\Gamma (e_i \cdot n) \, w_{e_i}^{a^\star I_d, a^\infty_3 I_d},
\end{equation}
where, for any $p \in \RR^d$, $w_p^{a^\star I_d, a^\infty_3 I_d}$ is the unique solution in $V_0$ to
\begin{equation}
\label{eq:lim2_bis}
-\mbox{\div}\Big( \left(a^\star \chi_B + a^\infty_3( 1- \chi_B) \right) \left(p + \nabla w_p^{a^\star I_d, a^\infty_3 I_d} \right)\Big) = 0 \mbox{ in $\cD'(\R^d)$}.
\end{equation}
Using the relation~\eqref{eq:rel1} for problem~\eqref{eq:lim2_bis}, we have that
\begin{multline*}
a^\star \int_B \left| \nabla w_p^{a^\star I_d, a^\infty_3 I_d} \right|^2 + a^\infty_3 \int_{\R^d \setminus B} \left| \nabla w_p^{a^\star I_d, a^\infty_3 I_d} \right|^2 
\\
= - a^\star \int_B p^T \nabla w_p^{a^\star I_d, a^\infty_3 I_d} + a^\infty_3 \int_\Gamma (p \cdot n) \, w_p^{a^\star I_d, a^\infty_3 I_d}.
\end{multline*}
We thus deduce from~\eqref{eq:genial2_bis} that
$$
d \, a^\infty_3 = d \, a^\star - \frac{1}{|B|} \sum_{i=1}^d \left( a^\star \int_B \left| \nabla w_{e_i}^{a^\star I_d, a^\infty_3 I_d} \right|^2 + a^\infty_3 \int_{\R^d \setminus B} \left| \nabla w_{e_i}^{a^\star I_d, a^\infty_3 I_d} \right|^2 \right).
$$
This implies that 
$$
a^\infty_3 \leq a^\star.
$$
The sequel of the proof follows the same lines as the proof of Proposition~\ref{prop:prop2}.

\section{Two special cases}
\label{sec:justif}

In this section, we consider two special cases: the one-dimensional case (in Section~\ref{sec:1D}) and the case of a homogeneous material (in Section~\ref{sec:homogene}). In the first case, we show that our three definitions yield the same approximation of $A^\star$ as the standard method based on~\eqref{eq:correctorrandom-N}. In the second case, we show that our three definitions yield the value of the homogeneous material.

\subsection{The one-dimensional case}
\label{sec:1D}

If $d=1$, then the solution to~\eqref{eq:pbbase} can be analytically computed. It satisfies
$$
\frac{d w^{\cA,A}}{dx} = 0 \quad \text{on $\R \setminus B$}, \qquad \frac{d w^{\cA,A}}{dx} = \frac{A}{\cA(x)}-1 \quad \text{on $B$}. 
$$
We then get from~\eqref{eq:expJ3} that
\begin{eqnarray}
\nonumber
|B| \, \cJ^{\cA^R}(A) 
&=&
\int_B \cA^R 
- 
\int_B \cA^R \left( \frac{d w^{\cA^R,A}}{dx} \right)^2 
- 
\int_{\R \setminus B} A \left( \frac{d w^{\cA^R,A}}{dx} \right)^2
\\
\nonumber
&=& 
\int_B \cA^R 
-
\left[ \frac{A^2 \, |B|}{\cA^\star_R} - 2 A \, |B| + \int_B \cA^R \right]
\\
\nonumber
&=& 
|B| \left( 2A - \frac{A^2}{A^\star_R} \right),
\label{eq:titi}
\end{eqnarray}
where we have introduced $\dps (A^\star_R)^{-1} := \frac{1}{|B|} \int_B (\cA^R)^{-1}$, namely the harmonic mean of $\cA^R$ on $B$. The definitions~\eqref{eq:optimisation}, \eqref{eq:def2} and~\eqref{eq:def3} all yield
$$
A^R_1 = A^R_2 = A^R_3 = A^\star_R.
$$

\medskip

We point out that, in this one-dimensional case, the approximate coefficient $A^{\star}_R$ is identical to the effective coefficient $A^{\star,R}$ defined by~\eqref{eq:defper} (i.e. considering a {\em truncated} corrector problem supplied with periodic boundary conditions). Thus, in this context, we can see that our alternative definitions of effective coefficients are consistent with the standard one. 

\subsection{The case of a homogeneous material}
\label{sec:homogene}

We assume here that
\begin{equation}
\label{eq:hyp_homog}
\text{for all $R>0$, \ \ $\cA^R = \cA$ is constant and equal to some matrix $\overline{A}\in\cM$.}
\end{equation} 
We show below that $\overline{A}$ is the unique maximizer of $\dps A \mapsto \sum_{i=1}^d \cJ_{e_i}^{\cA}(A)$, and hence that Definition~\eqref{eq:optimisation} yields $A^R_1 = \overline{A}$ for all $R>0$. We next show that Definition~\eqref{eq:def2} yields $A^R_2 = \overline{A}$. We eventually show that $A^R_3 = \overline{A}$ satisfies~\eqref{eq:def3}, and that $G^{\cA}$ has a unique fixed point. 

\subsubsection{Definition~\eqref{eq:optimisation}}

From~\eqref{eq:expJ3} and our assumption~\eqref{eq:hyp_homog}, we see that, for any $A \in \cM$,
\begin{equation}
\label{eq:titi2}
\cJ_p^{\cA}(A) \leq \frac{1}{|B|} \int_B p^T \cA p = p^T \, \overline{A} \, p,
\end{equation}
hence
$$
\sum_{i=1}^d \cJ_{e_i}^{\cA}(A) \leq \sum_{i=1}^d e_i^T \, \overline{A} \, e_i.
$$
If $A=\overline{A}$, we see that the diffusion matrix in~\eqref{eq:pbbase} is constant, therefore $w_p^{\cA,\overline{A}} = 0$. We then deduce from~\eqref{eq:optim2} that $\dps \cJ_p^{\cA}\left( \overline{A} \right) = p^T \, \overline{A} \, p$, which directly implies that $A = \overline{A}$ is a maximizer of $\dps \cM \ni A \mapsto \sum_{i=1}^d \cJ_{e_i}^{\cA}(A)$.

Conversely, assume that $\widehat{A}$ is a maximizer of $\dps \cM \ni A \mapsto \sum_{i=1}^d \cJ_{e_i}^{\cA}(A)$. Then, for any $1\leq i\leq d$, $\dps \cJ_{e_i}^{\cA}\left( \widehat{A} \right) = e_i^T \, \overline{A} \, e_i$. We thus infer from~\eqref{eq:expJ3} that $\nabla w_{e_i}^{\cA,\widehat{A}} = 0$. Using~\eqref{eq:pbbase}, we deduce that 
$$
\mbox{\div}\Big( \left( \overline{A} \chi_B + \widehat{A} \chi_{\RR^d \setminus B_R} \right) e_i \Big) = \mbox{\div}\Big( \ccA^{\cA, \widehat{A}} e_i \Big) = 0.
$$
Using Lemma~\ref{lem:tech}, we obtain that $\widehat{A} = \overline{A}$. 

\medskip

We hence have shown that $\overline{A}$ is the unique maximizer of $\dps \cM \ni A \mapsto \sum_{i=1}^d \cJ_{e_i}^{\cA}(A)$. Our first definition therefore yields $A^R_1 = \overline{A}$ for all $R>0$. 

\subsubsection{Definition~\eqref{eq:def2}}

We deduce from~\eqref{eq:def2}, the fact that $A^R_1 = \overline{A}$ and the above expression of $\cJ_p^{\cA}\left( \overline{A} \right)$ that, for any $p \in \R^d$,
$$
p^T A_2^R p = \cJ_p^{\cA}(A^R_1) = \cJ_p^{\cA}\left( \overline{A} \right) = p^T \, \overline{A} \, p.
$$
Since $A^R_2$ and $\overline{A}$ are symmetric, this implies that $A^R_2 = \overline{A}$.

\subsubsection{Definition~\eqref{eq:def3}}

We deduce from the above expression of $\cJ_p^{\cA}\left( \overline{A} \right)$ that $G^{\cA}\left( \overline{A} \right) = \overline{A}$, hence $\overline{A}$ is a fixed point of $G^{\cA}$. The remainder of this section is devoted to showing that $\overline{A}$ is the {\em unique} fixed point of $G^{\cA}$. We recast~\eqref{eq:titi2} as
$$
\forall p \in \R^d, \quad p^T G^{\cA}(A) p \leq p^T \, \overline{A} \, p.
$$
If $A$ is a fixed point of $G^{\cA}$, then we have that 
\begin{equation}
\label{eq:genial7}
A \leq \overline{A}.
\end{equation}
We now follow the same steps as in the proof of Proposition~\ref{prop:prop2}. Using~\eqref{eq:defGA} and~\eqref{eq:expJ2}, we see that 
\begin{eqnarray*}
  p^T A p
  =
p^T G^{\cA}(A) p 
&=& 
\cJ^{\cA}_p(A)
\\
&=& 
\frac{1}{|B|} \int_B p^T \cA (p + \nabla w_p^{\cA, A}) - \frac{1}{|B|} \int_{\Gamma} (A p\cdot n) \, w_p^{\cA, A}.
\end{eqnarray*}
Using~\eqref{eq:rel1}, we deduce that
$$
p^T A p = p^T \, \overline{A} p - \frac{1}{|B|} \int_B (\nabla w_p^{\cA, A})^T \cA \nabla w_p^{\cA, A} - \frac{1}{|B|} \int_{\R^d \setminus B} (\nabla w_p^{\cA, A})^T A \nabla w_p^{\cA, A}.
$$
We infer from the above relation and~\eqref{eq:rel1} that
\begin{eqnarray}
\int_B (p + \nabla w_p^{\cA, A})^T \cA p 
&=& 
|B| \, p^T \, \overline{A} p + \int_B p^T \cA \nabla w_p^{\cA, A}
\nonumber
\\
&=&
|B| \, p^T A p + \int_B (p + \nabla w_p^{\cA, A})^T \cA \nabla w_p^{\cA, A} + \int_{\R^d \setminus B} (\nabla w_p^{\cA, A})^T A \nabla w_p^{\cA, A}
\nonumber
\\
&=& 
|B| \, p^T A p  + \int_{\Gamma} (A p \cdot n) \, w_p^{\cA, A}.
\label{eq:genial4_bis}
\end{eqnarray}
The equality~\eqref{eq:rel1} can also be written as
$$
0 
=
\int_B (p + \nabla w_p^{\cA, A})^T \cA \nabla w_p^{\cA, A} - \int_{\Gamma} (A p\cdot n) \, w_p^{\cA, A} + \int_{\R^d \setminus B} (\nabla w_p^{\cA, A})^T A \nabla w_p^{\cA, A}.
$$
Subtracting~\eqref{eq:genial4_bis} from the above relation and next using~\eqref{eq:genial7}, we get
\begin{eqnarray}
0 
&=& \int_B (p + \nabla w_p^{\cA, A})^T \cA (p + \nabla w_p^{\cA, A}) - 2 \int_{\Gamma} (A p\cdot n) \, w_p^{\cA, A}
\nonumber
\\
& & \qquad + \int_{\R^d \setminus B} (\nabla w_p^{\cA, A})^T A \nabla w_p^{\cA, A} - |B| \, p^T A p 
\nonumber
\\
& \geq & 
\int_B (p + \nabla w_p^{\cA, A})^T A (p + \nabla w_p^{\cA, A}) - 2 \int_{\Gamma} (A p\cdot n) \, w_p^{\cA, A}
\nonumber
\\
& & \qquad + \int_{\R^d \setminus B} (\nabla w_p^{\cA, A})^T A \nabla w_p^{\cA, A} - |B| \, p^T A p.
\label{eq:genial6_bis}
\end{eqnarray}
We now define, for all $v\in V_0$, the energy
$$
\cI(v) := \frac{1}{2} \int_B (p + \nabla v)^T A (p + \nabla v) - \int_{\Gamma} (A p\cdot n) \, v + \frac{1}{2} \int_{\R^d \setminus B} (\nabla v)^T A \nabla v - \frac{1}{2} |B| \, p^T A p.
$$
The unique solution $v_0 \in V_0$ to the minimization problem
$$
v_0 = \mathop{\mbox{argmin}}_{v\in V_0} \cI(v)
$$
satisfies
$$
-\mbox{\div} \Big( A (p + \nabla v_0) \Big) = 0 \mbox{ in $\cD'(R^d)$},
$$  
and therefore is simply $v_0 = 0$. Thus,
\begin{equation}
\label{eq:genial5_bis}
\forall v \in V_0, \quad \cI(v) \geq \cI(v_0) = 0.
\end{equation}
We recast~\eqref{eq:genial6_bis} as
$$
0 \geq 2 \, \cI(w_p^{\cA, A}) \quad \text{with $w_p^{\cA, A} \in V_0$}.
$$ 
Collecting the above relation with~\eqref{eq:genial5_bis}, we deduce that $w_p^{\cA, A}$ is the unique minimizer of $\cI$ on $V_0$, hence that $w_p^{\cA, A} = 0$. This results holds for all $p\in\R^d$. In view of Lemma~\ref{lem:tech} and our assumption~\eqref{eq:hyp_homog}, we thus obtain that $A = \overline{A}$. This is the claimed uniqueness result of the fixed point of $G^{\cA}$, under assumption~\eqref{eq:hyp_homog}.

\section*{Acknowledgements}

The work of FL is partially supported by ONR under Grant N00014-15-1-2777 and EOARD under grant FA9550-17-1-0294. SX gratefully acknowledges the support from Labex MMCD (Multi-Scale Modelling \& Experimentation of Materials for Sustainable Construction) under contract ANR-11-LABX-0022. 
The authors acknowledge the funding from the German Academic Exchange Service (DAAD) from funds of the ``Bundesministeriums f\"ur Bildung und Forschung'' (BMBF) for the project Aa-Par-T (Project-ID 57317909) as well as the funding from the PICS-CNRS and the PHC PROCOPE 2017 (Project N$^0$ 37855ZK). The authors would also like to thank S\'ebastien Brisard for useful discussions on a preliminary version of this work.

\bibliography{biblio}

\end{document}

%% file: random3.pstex_t
\begin{picture}(0,0)%
\includegraphics{random3.pstex}%
\end{picture}%
\setlength{\unitlength}{2072sp}%
\begingroup\makeatletter\ifx\SetFigFont\undefined%
\gdef\SetFigFont#1#2#3#4#5{%
  \reset@font\fontsize{#1}{#2pt}%
  \fontfamily{#3}\fontseries{#4}\fontshape{#5}%
  \selectfont}%
\fi\endgroup%
\begin{picture}(9834,8934)(1114,-8713)
\put(2251,-2536){\makebox(0,0)[lb]{\smash{{\SetFigFont{14}{16.8}{\rmdefault}{\mddefault}{\updefault}{\color[rgb]{0,0,0}$\mathbb{A}(x,\omega)$}%
}}}}
\put(8326,-1366){\makebox(0,0)[lb]{\smash{{\SetFigFont{25}{30.0}{\rmdefault}{\mddefault}{\updefault}{\color[rgb]{0,0,0}$B_R$}%
}}}}
\put(7021,-7711){\makebox(0,0)[lb]{\smash{{\SetFigFont{14}{16.8}{\rmdefault}{\mddefault}{\updefault}{\color[rgb]{0,0,0}$R$}%
}}}}
\end{picture}%

%% file: random5.pstex_t
\begin{picture}(0,0)%
\includegraphics{random5.pstex}%
\end{picture}%
\setlength{\unitlength}{2072sp}%
\begingroup\makeatletter\ifx\SetFigFont\undefined%
\gdef\SetFigFont#1#2#3#4#5{%
  \reset@font\fontsize{#1}{#2pt}%
  \fontfamily{#3}\fontseries{#4}\fontshape{#5}%
  \selectfont}%
\fi\endgroup%
\begin{picture}(9834,8934)(1114,-8713)
\put(8326,-1366){\makebox(0,0)[lb]{\smash{{\SetFigFont{25}{30.0}{\rmdefault}{\mddefault}{\updefault}{\color[rgb]{0,0,0}$B_R$}%
}}}}
\put(7021,-7711){\makebox(0,0)[lb]{\smash{{\SetFigFont{14}{16.8}{\rmdefault}{\mddefault}{\updefault}{\color[rgb]{0,0,0}$R$}%
}}}}
\put(5041,-4021){\makebox(0,0)[lb]{\smash{{\SetFigFont{14}{16.8}{\rmdefault}{\mddefault}{\updefault}{\color[rgb]{0,0,0}$\mathbb{A}(x,\omega)$}%
}}}}
\put(2251,-1456){\makebox(0,0)[lb]{\smash{{\SetFigFont{14}{16.8}{\rmdefault}{\mddefault}{\updefault}{\color[rgb]{0,0,0}$A$}%
}}}}
\end{picture}%

%% file: cances_et_al.bbl
\begin{thebibliography}{10}

\bibitem{Allaire}
G.~Allaire.
\newblock {\em Shape Optimization by the Homogenization Method}.
\newblock Springer, 2002.

\bibitem{singapour}
A.~Anantharaman, R.~Costaouec, C.~Le~Bris, F.~Legoll, and F.~Thomines.
\newblock Introduction to numerical stochastic homogenization and the related
  computational challenges: some recent developments.
\newblock In W.~Bao and Q.~Du, editors, {\em Multiscale modeling and analysis
  for materials simulation}, volume~22 of {\em Lect. Notes Series, Institute
  for Mathematical Sciences, National University of Singapore}, pages 197--272.
  World Sci. Publ., Hackensack, NJ, 2011.

\bibitem{Bensoussan}
A.~Bensoussan, J.-L. Lions, and G.~Papanicolaou.
\newblock {\em Asymptotic Analysis for Periodic Structures}.
\newblock American Mathematical Society, 1978.

\bibitem{Benveniste}
Y.~Benveniste.
\newblock A new approach to the application of {M}ori-{T}anaka's theory in
  composite materials.
\newblock {\em Mechanics of Materials}, 6:147--157, 1987.

\bibitem{filtrage_sab}
F.~Bignonnet, K.~Sab, L.~Dormieux, S.~Brisard, and A.~Bisson.
\newblock Macroscopically consistent non-local modeling of heterogeneous media.
\newblock {\em Computer Methods in Applied Mechanics and Engineering},
  278:218--238, 2014.

\bibitem{filtrage_blanc}
X.~Blanc and C.~Le~{B}ris.
\newblock Improving on computation of homogenized coefficients in the periodic
  and quasi-periodic settings.
\newblock {\em Net. Heterog. Media}, 5:1--29, 2010.

\bibitem{Bourgeat}
A.~Bourgeat and A.~Piatniski.
\newblock An optimal error estimate in stochastic homogenization of discrete
  elliptic equations.
\newblock {\em Ann. I. H. Poincaré}, 40:153--165, 2004.

\bibitem{brisard_autres_CL}
S.~Brisard, K.~Sab, and L.~Dormieux.
\newblock New boundary conditions for the computation of the apparent stiffness
  of statistical volume elements.
\newblock {\em J. Mech. Phys. Solids}, 61:2638--2658, 2013.

\bibitem{notre_cras}
E.~Canc\`es, V.~Ehrlacher, F.~Legoll, and B.~Stamm.
\newblock An embedded corrector problem to approximate the homogenized
  coefficients of an elliptic equation.
\newblock {\em C. R. Acad. Sci. Paris, Série I}, 353:801--806, 2015.

\bibitem{refpartii}
E.~Canc\`es, V.~Ehrlacher, F.~Legoll, B.~Stamm, and S.~Xiang.
\newblock An embedded corrector problem for homogenization. {P}art {II}:
  {A}lgorithms and discretization.
\newblock {\em In preparation}, 2018.

\bibitem{Stamm1}
E.~Canc\`es, Y.~Maday, and B.~Stamm.
\newblock Domain decomposition for implicit solvation models.
\newblock {\em Journal of Chemical Physics}, 139:054111, 2013.

\bibitem{Christensen}
R.M. Christensen and K.H. Lo.
\newblock Solutions for effective shear properties in three sphere and cylinder
  models.
\newblock {\em J. Mech. Phys. Solids}, 27:315--330, 1979.

\bibitem{Cioranescu}
D.~Cioranescu and P.~Donato.
\newblock {\em An Introduction to Homogenization}.
\newblock Oxford University Press, New York, 1999.

\bibitem{cottereau}
R.~Cottereau.
\newblock Numerical strategy for unbiased homogenization of random materials.
\newblock {\em Int. J. Numer. Methods Eng.}, 95:71--90, 2013.

\bibitem{engquist-souganidis}
B.~Engquist and P.E. Souganidis.
\newblock Asymptotic and numerical homogenization.
\newblock {\em Acta Numerica}, 17:147--190, 2008.

\bibitem{Eshelby}
J.D. Eshelby.
\newblock The determination of the elastic field of an ellipsoidal inclusion,
  and related problems.
\newblock {\em Proceedings of the Royal Society of London. Series A,
  Mathematical and Physical Sciences}, 241:376--396, 1957.

\bibitem{gilbarg2001elliptic}
D.A. Gilbarg and N.S. Trudinger.
\newblock {\em Elliptic partial differential equations of second order}, volume
  224.
\newblock Springer, 2001.

\bibitem{GloriaOtto1}
A.~Gloria and F.~Otto.
\newblock An optimal variance estimate in stochastic homogenization of discrete
  elliptic equations.
\newblock {\em Ann. Probab.}, 39:779--856, 2011.

\bibitem{Hill}
R.~Hill.
\newblock A self-consistent mechanics of composite materials.
\newblock {\em J. Mech. Phys. Solids}, 13:213--222, 1965.

\bibitem{Huet}
C.~Huet.
\newblock Application of variational concepts to size effects in elastic
  heterogeneous bodies.
\newblock {\em J. Mech. Phys. Solids}, 38:813--841, 1990.

\bibitem{Jikov}
V.~Jikov, S.~Kozlov, and O.~Oleinik.
\newblock {\em Homogenization of differential operators and integral
  functionals}.
\newblock Springer, Berlin, 1995.

\bibitem{kozlov}
S.M. Kozlov.
\newblock Averaging of random operators.
\newblock {\em Matematicheskii Sbornik}, 151(2):188--202, 1979.

\bibitem{Birkhoff1}
U.~Krengel.
\newblock {\em Ergodic theorems}, volume~6 of {\em Studies in Mathematics}.
\newblock de Gruyter, 1985.

\bibitem{legoll_jcp}
C.~Le~{B}ris and F.~Legoll.
\newblock Examples of computational approaches for elliptic, possibly
  multiscale {PDEs} with random inputs.
\newblock {\em J. Comput. Physics}, 328:455--473, 2017.

\bibitem{lemaire}
C.~Le~{B}ris, F.~Legoll, and S.~Lemaire.
\newblock On the best constant matrix approximating an oscillatory
  matrix-valued coefficient in divergence-form operators.
\newblock {\em Control, Optimisation and Calculus of Variations}, 2018.
\newblock to appear (available at {\tt https://arxiv.org/abs/1612.05807}).

\bibitem{Stamm2}
F.~Lipparini, B.~Stamm, E.~Cancès, Y.~Maday, and B.~Mennucci.
\newblock A fast domain decomposition algorithm for continuum solvation models:
  Energy and first derivatives.
\newblock {\em J. Chem. Theory Comput.}, 9:3637--3648, 2013.

\bibitem{lu_otto}
J.~Lu and F.~Otto.
\newblock Optimal artificial boundary condition for random elliptic media.
\newblock {\em available at {\tt http://arxiv.org/pdf/1803.09593}}, 2018.

\bibitem{MoriTanaka}
T.~Mori and K.~Tanaka.
\newblock Average stress in matrix and average elastic energy of materials with
  misfitting inclusions.
\newblock {\em Acta Metallurgica}, 21:571--574, 1973.

\bibitem{mourrat}
J.-C. Mourrat.
\newblock Efficient methods for the estimation of homogenized coefficients.
\newblock {\em Found. Comp. Math.}, 2018.
\newblock to appear.

\bibitem{MuratTartar}
F.~Murat and L.~Tartar.
\newblock H-convergence.
\newblock {\em S\'eminaire d'Analyse Fonctionnelle et Numérique de
  l'Universit\'e d'Alger}, 1978.

\bibitem{nolen}
J.~Nolen.
\newblock Normal approximation for a random elliptic equation.
\newblock {\em Probab. Theory Related Fields}, 159:661--700, 2014.

\bibitem{papa}
G.C. Papanicolaou and S.R.S. Varadhan.
\newblock Boundary value problems with rapidly oscillating random coefficients.
\newblock In J.~Fritz, J.L. Lebaritz, and D.~Szasz, editors, {\em Proc. Colloq.
  on Random fields: Rigorous results in statistical mechanics and quantum field
  theory}, volume~10 of {\em Colloq. Math. Soc. J\'anos Bolyai}, pages
  835--873. North-Holland, Amsterdam-New York, 1981.

\bibitem{protter}
M.H. Protter.
\newblock Unique continuation for elliptic equations.
\newblock {\em Trans. Amer. Math. Soc.}, 95:81--91, 1960.

\bibitem{Birkhoff2}
A.N. Shiryaev.
\newblock {\em Probability}, volume~95 of {\em Graduate Texts in Mathematics}.
\newblock Springer, 1984.

\bibitem{Birkhoff3}
A.A. Tempel'man.
\newblock Ergodic theorems for general dynamical systems.
\newblock {\em Trudy Moskov. Mat. Obsc.}, 26:94--132, 1972.

\bibitem{Thomines}
F.~Thomines.
\newblock {\em M\'ethodes math\'ematiques et techniques num\'eriques de
  changement d'\'echelle: application aux mat\'eriaux al\'eatoires}.
\newblock PhD thesis, Universit\'e Paris-Est, 2012.

\end{thebibliography}
